\documentclass[12pt, twoside, leqno]{article}

\usepackage{amsmath,amsthm}
\usepackage{amssymb}
\usepackage{cases,mathrsfs}%%added by authors

\usepackage{enumerate}

\newtheorem{theorem}{Theorem}[section]
\newtheorem{lemma}[theorem]{Lemma}
\newtheorem{proposition}[theorem]{Proposition}

\theoremstyle{definition}
\newtheorem{definition}[theorem]{Definition}

\numberwithin{equation}{section}

\begin{document}

%%%%% To ease editing, for IMPAN journals add:

\baselineskip=17pt

%%%%%%%%%%%%%%%%

\title{Sharp Weighted convolution inequalities and some applications}

\author{
WEICHAO GUO\\
School of Mathematics and Information Sciences\\
Guangzhou University\\
Guangzhou, 510006, P.R.China\\
E-mail: weichaoguomath@gmail.com
\and
DASHAN FAN\\
Department of Mathematics\\
University of Wisconsin-Milwaukee\\
Milwaukee, WI 53201, USA\\
E-mail: fan@uwm.edu
\and
HUOXIONG WU\\
School of Mathematical Sciences\\
Xiamen University\\
Xiamen, 361005, P.R.China\\
E-mail: huoxwu@xmu.edu.cn
\and
GUOPING ZHAO\\
School of Applied Mathematics\\
Xiamen University of Technology\\
Xiamen, 361024, P.R.China\\
E-mail: guopingzhaomath@gmail.com
}

\date{}

\maketitle

%% Classification and key words; note that the 2010 classification is used:

\renewcommand{\thefootnote}{}



\renewcommand{\thefootnote}{\arabic{footnote}}
\setcounter{footnote}{0}

%%%%%%%%

\begin{abstract}
In this paper, the index groups for which the weighted Young's inequalities hold in both continuous case and discrete case are characterized.
As applications, the index groups for the product inequalities on modulation spaces are characterized, we also
obtain the weakest conditions for the boundedness of bilinear Fourier multipliers on modulation spaces in some sense.
For the fractional integral operator, the sharp conditions
for the boundedness of power weighted $L^p-L^q$ estimates in both continuous case and discrete case are obtained.
By a unified approach different from others, we complete some previous results which are committed to finding sharp conditions for
some classical inequalities.
\end{abstract}

\section{Introduction}
This paper is devoted to studying some convolution inequalities on weighted Lebesgue spaces, for both discrete and continuous versions,
with the aim of finding the sharp conditions for boundedness about the Young convolution inequality and the fractional integral operator on these function spaces.

The convolution of two measurable functions on $\mathbb{R}^n$ is defined by
\begin{equation*}
(f\ast g)(x)=\int_{\mathbb{R}^n}f(x-y)g(y)dy,~x\in\mathbb{R}^n.
\end{equation*}
If $1\leq p,q,r \leq \infty$ and $1+1/q=1/p+1/r$, then it is well known that the classical Young inequality
\begin{equation}\label{Introduction 1}
\|f\ast g\|_{L^q}\leq \|f\|_{L^p}\|g\|_{L^r}
\end{equation}
plays a fundamental role in studying the convolution operator.

In this paper, we adopt the notation
$L^p\ast L^r\subset L^q$ to denote (\ref{Introduction 1}), for the sake of simplicity.
More generally, for function spaces $X,Y$ and $Z$,
the expression of form $X\ast Y\subset Z$ means that whenever $f\in X$, $g\in Y$, then
$f\ast g\in Z$ and
\begin{equation}\label{Introduction 2}
\|f\ast g\|_Z\lesssim \|f\|_X\|g\|_Y.
\end{equation}
Inequalities of the form (\ref{Introduction 2}) are usually called the Young-type (convolution) inequalities.

In this paper, we will focus on the Young-type inequalities on the power weighted Lebesgue spaces.
Let \ $s$ be a real number and $1\leq p\leq \infty.$
We use $\mathbb{L}(p,s)$ to denote the weighted $L^p$ Lebesgue space with power weight $|x|^s$,
and $L(p,s)$ to denote the weighted $L^p$ Lebesgue space with power weight (without a singularity at the origin) $\langle x \rangle^s=(1+|x|^2)^{s/2}$.
Also, the space $l(p,s)$ denotes the discrete counterpart of $L(p,s)$.
From a technical point of view, if we ignore the possible singularity of the weight $|x|^s$ at the origin
as in Proposition \ref{Implication method}, $l(p,s)$ can be also regarded as the discrete counterpart of $\mathbb{L}(p,s)$.
In fact, the relationship between $l(p,s)$ and $\mathbb{L}(p,s)$ are quite important for our proof.
Since \ $L(p,0)=$ $\mathbb{L}(p,0)=L^{p},$ one easily expects an
immediate extension of the classical Young inequality: the inclusion
\begin{equation*}
\mathbb{L}(q_{1},s_{1})\ast \mathbb{L}(q_{2},s_{2})\subset \mathbb{L}(q,s)
\end{equation*}
holds for appropriate indices $q_{1},q_{2},q$ and $s_{1},s_{2},s.$\ To this
end, finding sharp conditions on these indices to ensure the Young
inequality
\begin{equation*}
\Vert f\ast g\Vert _{\mathbb{L}(q,s)}\precsim \Vert f\Vert _{\mathbb{L}
(q_{1},s_{1})}\Vert g\Vert _{\mathbb{L}(q_{2},s_{2})}
\end{equation*}
is considerable and interesting, and this inequality and its varieties might
play a pivotal role when we study the convolution operators in the weighted
Lebesgue spaces. This problem of course motivated a lot of research works.
In the following we briefly review the historical development, by listing a
few of research articles related to the topic in this paper.

The study of Young's inequality on the spaces $\mathbb{L}(p,s)$ can be dated
back as early as thirty years ago. In 1983, Kerman obtained the following
theorem.
\\
\textbf{Theorem A} (Kerman \cite{Kerman}).\label{Kerman's result}
Let $1\leq q,q_1,q_2\leq \infty$, $s,s_1,s_2\in \mathbb{R}.$
Suppose that $(\mathbf{q},\mathbf{s})=(q, q_1, q_2, s, s_1, s_2)$ satisfies
\begin{equation*}
\begin{cases}
s\leq s_1,~s\leq s_2,~0\leq s_1+s_2,\\
1+\frac{1}{q}+\frac{s}{n}=\frac{1}{q_1}+\frac{s_1}{n}+\frac{1}{q_2}+\frac{s_2}{n},~\frac{1}{q}\leq \frac{1}{q_1}+\frac{1}{q_2},\\
\frac{1}{q}+\frac{s}{n}<\frac{1}{q_1}+\frac{s_1}{n},~\frac{1}{q}+\frac{s}{n}<\frac{1}{q_2}+\frac{s_2}{n},~\frac{1}{q}+\frac{s}{n}>0,\\
q\neq \infty,~q_1,q_2\neq 1.
\end{cases}
\end{equation*}
Then
\begin{equation}\label{Introduction 3}
\mathbb{L}(q_1, s_1)\ast \mathbb{L}(q_2, s_2) \subset \mathbb{L}(q, s).
\end{equation}

The above convolution inequality was also studied by Bui \cite{Bui}, among other
authors. Further weighted inequalities for convolutions can be found in
\cite{Biswas_Swanson, Kerman_Sawyer, NT_jfaa}.
In \cite{Bui}, Bui obtained some necessary conditions
for the inclusion (\ref{Introduction 3}). However, these necessary conditions
are not matched the sufficient conditions in Theorem A. Bui thus posed the
question for finding the sharp conditions on (\ref{Introduction 3}). This
question was solved just very recently by Nursultanov and Tikhonov \cite{NT}
in the ranges $1<q_{1},q_{2}<\infty $ and $1<q\leq \infty $, but with an
extra assumption ${1}/{q}\leq {1}/{q_{1}}+{1}/{q_{2}}$. We note
that the result of Nursultanov and Tikhonov does not imply the end point
cases $q_{1},\,q_{2}=1,\infty $ \ and $q=1,$ while these cases sometime are
notably important in applications. Also, the extra condition  ${1}/{q}
\leq {1}/{q_{1}}+{1}/{q_{2}}$ \ seems little odd. Therefore, based
on these observations, in this paper we will give a complete answer to Bui's
question by establishing sharp (sufficient and necessary) conditions of (\ref{Introduction 3}) in the full ranges $1\leq q_{1},q_{2}\leq \infty $ and $%
1\leq q\leq \infty $. More significantly, our result removes the extra
assumption ${1}/{q}\leq {1}/{q_{1}}+{1}/{q_{2}}$ (this condition
actually, in many cases, is implicitly contained in the necessary
conditions).

Since the method used by Nursultanov and Tikhonov is based on an extra
assumption and it also raises some difficulties to treat the end point
cases, in this paper we will use a quite different approach. We first study
the convolution inequalities in the discrete weighted Lebesgue spaces $%
l(q,s) $. Then we reduce the continuous case to the discrete one to reach
our target. On the other hand, we find that the convolution inequalities in
the discrete case itself is of interest. We will show that the discrete form
of weighted convolution inequality not only has a closed relation to its continuous
counterpart, but also is a powerful tool to study the algebraic property of
the modulation spaces (see Theorem 1.4).

We also notice that a recent paper \cite{Toft} also addresses the Young inequality on the spaces  $L(p,s).$
The authors establish some sufficient conditions on  $L(q_1, s_1)\ast L(q_2, s_2) \subset L(q, s)$.
They also find some partial necessary conditions. However, there is a big distance between sufficiency and necessity.
Again, their methods are different from ours.

As a conclusion, in the full range $1\leq q,\,q_{1},\,q_{2}\leq \infty $, using
different methods from others we will find the sharp conditions for the
convolution inequalities in both discrete and continuous weighted Lebesgue
spaces.

Let us first list 4 important relations among the indices $q,\,q_{1},\,q_{2},$
and $\ s,\,s_{1},\,s_{2}.$

\begin{eqnarray}
&&(\mathcal {A}_1)\begin{cases}
s\leq s_1,~s\leq s_2,~0\leq s_1+s_2,\\
1+\Big(\frac{1}{q}+\frac{s}{n}\Big)\vee 0<\Big(\frac{1}{q_1}+\frac{s_1}{n}\Big)\vee 0+\Big(\frac{1}{q_2}+\frac{s_2}{n}\Big)\vee 0,\\
\frac{1}{q}+\frac{s}{n}\leq \frac{1}{q_1}+\frac{s_1}{n},~\frac{1}{q}+\frac{s}{n}\leq \frac{1}{q_2}+\frac{s_2}{n},
1\leq \frac{1}{q_1}+\frac{s_1}{n}+\frac{1}{q_2}+\frac{s_2}{n},\\
(q,s)=(q_1,s_1) ~\text{if}~ \frac{1}{q}+\frac{s}{n}=\frac{1}{q_1}+\frac{s_1}{n}, \\
(q,s)=(q_2,s_2) ~\text{if}~ \frac{1}{q}+\frac{s}{n}=\frac{1}{q_2}+\frac{s_2}{n},\\
(q'_1, -s_1)=(q_2,s_2)~\text{if}~ 1=\frac{1}{q_1}+\frac{s_1}{n}+\frac{1}{q_2}+\frac{s_2}{n};
\end{cases}
\\
&&(\mathcal {A}_2)\begin{cases}
s=s_1=s_2=0,\\
q=q_1, q_2=1~or~q=q_2, q_1=1~or~q=\infty, \frac{1}{q_1}+\frac{1}{q_2}=1;
\end{cases}
\\
&&(\mathcal {A}_3)\begin{cases}
s\leq s_1,~s\leq s_2,\\
\frac{1}{q_1}+\frac{1}{q_2}=1,~s_1+s_2=0,\\
\frac{1}{q}+\frac{s}{n}<0\leq \frac{1}{q_1}+\frac{s_1}{n},\frac{1}{q_2}+\frac{s_2}{n};
\end{cases}
\end{eqnarray}
\begin{eqnarray}
&&(\mathcal {A}_4)\begin{cases}
s\leq s_1,~s\leq s_2,~0\leq s_1+s_2,\\
1+\frac{1}{q}+\frac{s}{n}=\frac{1}{q_1}+\frac{s_1}{n}+\frac{1}{q_2}+\frac{s_2}{n},~\frac{1}{q}\leq \frac{1}{q_1}+\frac{1}{q_2},\\
\frac{1}{q}+\frac{s}{n}<\frac{1}{q_1}+\frac{s_1}{n},~\frac{1}{q}+\frac{s}{n}<\frac{1}{q_2}+\frac{s_2}{n},~\frac{1}{q}+\frac{s}{n}>0,\\
q\neq \infty,~q_1,q_2\neq 1,~\text{if}~s=s_1~or~s=s_2.
\end{cases}
\end{eqnarray}
Throughout this paper, we use $p'$ to denote the dual index of $p$ such that $1/p+1/p'=1$, and
use the notation \
\begin{equation*}
a\vee b=\max \{a,b\}.
\end{equation*}

\bigskip

Now, we state our main results associated with convolution inequalities on
weighted Lebesgue spaces.

\begin{theorem}[Young's inequality, discrete form, weight $\langle k\rangle$] \label{Sharpness of weighted Young's inequality, discrete form}
Suppose $1\leq q, q_1, q_2\leq \infty$, $s, s_1, s_2\in \mathbb{R}$. Then
\begin{equation}\label{Theorem 1, conclusion}
l(q_1, s_1)\ast l(q_2, s_2) \subset l(q, s)
\end{equation}
if and only if $(\mathbf{q},\mathbf{s})=(q, q_1, q_2, s, s_1, s_2)$ satisfies one of the conditions $\mathcal {A}_i$, $i=1,2,3,4$.
\end{theorem}

\begin{theorem}[Young's inequality, continuous, weight $|x|$] \label{Sharpness of weighted Young's inequality, power weight}
Suppose $1\leq q, q_1, q_2\leq \infty$, $s, s_1, s_2\in \mathbb{R}$. Then
\begin{equation*}
\mathbb{L}(q_1, s_1)\ast \mathbb{L}(q_2, s_2) \subset \mathbb{L}(q, s)
\end{equation*}
if and only if $(\mathbf{q},\mathbf{s})$ satisfies one of conditions $\mathcal {A}_2$ and $\mathcal {A}_4$.
\end{theorem}

\begin{theorem}[Young's inequality, continuous, weight $\langle x\rangle$] \label{Sharpness of weighted Young's inequality, continuous form}
Suppose $1\leq q, q_1, q_2\leq \infty$, $s, s_1, s_2\in \mathbb{R}$. Then
\begin{equation*}
L(q_1, s_1)\ast L(q_2, s_2) \subset L(q, s)
\end{equation*}
if and only if $(\mathbf{q},\mathbf{s})$ satisfies
\begin{equation*}
1+\frac{1}{q}\geq \frac{1}{q_1}+\frac{1}{q_2}
\end{equation*}
and one of the conditions $\mathcal {A}_i$, $i=1,2,3,4$.
\end{theorem}

As an application, we will study the product inequalities on the modulation
spaces. Then, as a consequence, we obtain an algebraic property for
modulation spaces, while it is known that this algebraic property is a key
issue to study certain nonlinear Cauchy problem of dissipative partial
differential equations on the modulation spaces \cite{WH_JDE_2007}. The
modulation space $M_{p,q}^{s}$ was introduced by Feichtinger \cite%
{Feichtinger} in 1983 by means of the short-time Fourier transform. Another
equivalent definition of $\ M_{p,q}^{s}$ \ can be given by applying the
frequency-uniform localizations (see \cite{WH_JDE_2007} for details). The
interested reader may find a lot of research articles, in the literature,
that address the space $M_{p,q}^{s}$ , as well as its many applications. For
instance, see \cite{Wang_book} for some basic properties of modulation
spaces, \cite{AKKL_JFA_2007, Feichtinger_Narimani} for the study of
boundedness on modulation spaces for certain operators. Particularly, it is
known that the modulation space serves as a good alternative working frame,
in many cases, in the study of partial differential equations, see \cite{Iwabuchi, WH_JDE_2007, RSW_2012}.
The definitions of modulation space will be presented in Section 2, but we would like to give the reader an earlier
notice that Theorem \ref{Sharpness of weighted Young's inequality, discrete
form} is a crucial inequality to obtain the product inequalities on
modulation spaces. More precisely, using Theorem \ref{Sharpness of weighted Young's
inequality, discrete form} we will establish the following algebraic
property of the modulation spaces.

\begin{theorem}[Product on modulation spaces] \label{Sharpness of product on modulation spaces, bilinear case}
Suppose $1\leq p, p_1, p_2\leq \infty$, $1\leq q, q_1, q_2\leq \infty$, $s, s_1, s_2\in \mathbb{R}$. Then
\begin{equation*}
\|fg\|_{M_{p,q}^s}\lesssim \|f\|_{M_{p_1,q_1}^{s_1}}\|g\|_{M_{p_2,q_2}^{s_2}}
\end{equation*}
holds for all $f,g\in \mathscr {S}$ if and only if
$1/p\leq 1/p_1+1/p_2$ and $(\mathbf{q},\mathbf{s})$ satisfies one of the conditions $\mathcal {A}_i$, $i=1,2,3,4$.
\end{theorem}

We note that a simpler case of the above theorem was obtained in \cite
{Cordero_Nicola}. However, using our method we are able to study a more
general bilinear Fourier multiplier $T(f,g)$ that takes the product \ $fg$ \ as a
special case. The bilinear Fourier multiplier $T$ with symbol \ $m(\xi ,\eta )$ \ is
defined on the product Schwartz space $\mathscr{S\times S}$ \ by
\begin{equation*}
T(f,g)(x)=\int_{
\mathbb{R}
^{n}}\int_{
\mathbb{R}
^{n}}m(\xi ,\eta )\widehat{f}(\xi )\widehat{g}(\eta )e^{2\pi i<x,\xi +\eta
>}d\xi d\eta
\end{equation*}
for all \ $f,g\in \mathscr{S}$. With the relation of the Fourier transform
and its inverse, it is easy to see that
\begin{equation*}
T(f,g)(x)=f(x)g(x)
\end{equation*}
if \ $m(\xi ,\eta )\equiv 1.$\

An interesting question is whether \ $T(f,g)$ \ is bounded on the modulation
spaces provided it is bounded on certain Lebesgue spaces. We find the sharp
conditions to answer this question.
\begin{theorem}[Multi-linear Fourier multipliers on modulation spaces] \label{Sharp conditions for the multi-linear Fourier multipliers on modulation spaces}
Suppose $1\leq q, q_1, q_2\leq \infty$, $s, s_1, s_2\in \mathbb{R}$. Then
\begin{equation*}
\|T(f,g)\|_{L^p}\lesssim \|f\|_{L^{p_1}}\|g\|_{L^{p_2}}\Longrightarrow \|T(f,g)\|_{M_{p,q}^s}\lesssim \|f\|_{M_{p_1,q_1}^{s_1}}\|g\|_{M_{p_2,q_2}^{s_2}}
\end{equation*}
holds for any bilinear Fourier multiplier $T$ and $1\leq p, p_1, p_2\leq \infty$, if and only if $(\mathbf{q},\mathbf{s})$ satisfies one of the conditions $\mathcal {A}_i$, $i=1,2,3,4$.
\end{theorem}

The convolution $f\ast g$ may be naturally regarded as a bilinear operator.
One may fix the function $g$ as the kernel function and consider the operator
\begin{equation*}
T_gf(x)=(f\ast g)(x).
\end{equation*}
In Fourier analysis, an important operator with this form is the fractional integral operator (or Riesz potential)
$I_{\lambda}$ defined by
\begin{equation*}
(I_{\lambda}f)(x)=\int_{\mathbb{R}^n}\frac{f(y)}{|x-y|^{n-\lambda}}dy,\hspace{10mm}0<\lambda<n.
\end{equation*}
The famous Hardy-Littlewood-Sobolev theorem gives the boundedness of $
I_{\lambda }$ from $L^{p}(\mathbb{R}^{n})$ to $L^{q}(\mathbb{R}^{n}),$
provided $1<p\leq q<\infty $ and ${1}/{q}={1}/{p}-{\lambda}/{n}.$
Continuity properties of the potential operator in the Lebesgue spaces
are well known, see \cite{Stein_Weiss_book, Grafakos_book}.
The following weighted version of the Hardy-Littlewood-Sobolev theorem was
obtained by Stein and Weiss five decades ago in \cite{Stein_Weiss}.
\\
\textbf{Theorem B} (Stein-Weiss \cite{Stein_Weiss}).
Suppose $1<p\leq q<\infty$, $s,t\in \mathbb{R}$. If
\begin{equation*}
\begin{cases}
s\leq t,
\\
\frac{1}{q}+\frac{s}{n}=\frac{1}{p}+\frac{t}{n}-\frac{\lambda}{n},\\
\frac{1}{p}+\frac{t}{n}<1, \frac{1}{q}+\frac{s}{n}>0,
\end{cases}
\end{equation*}
then
\begin{equation*}
I_{\lambda}: \mathbb{L}(p,t) \rightarrow \mathbb{L}(q,s).
\end{equation*}
A more general result on $I_{\lambda }:\mathbb{L}(p,t)\rightarrow \mathbb{L}
(q,s),$ including the endpoint $p=1$ or $q=\infty $ can be found in \cite
{Strichartz}, in which the author provided an alternative proof. Under the
assumption $1\leq p\leq q<\infty $, Duoandikoetxea \cite{Duoandikoetxea}
found some necessity conditions on the map $I_{\lambda }:\mathbb{L}
(p,t)\rightarrow \mathbb{L}(q,s).$ Also in the same paper, Duoandikoetxea
obtained the (partial) necessary conditions in the radial case. Recently,
Nowak and Stempak claimed the complete result in the radial case, by
finding the sharp conditions for $I_{\lambda }:\mathbb{L}(p,t)\rightarrow
\mathbb{L}(q,s),$ including the endpoint $q=\infty $ (see Corollary 2.6 in
\cite{Nowak_Krzysztof}).
Some work associated with the weighted inequality for fractional integral operator can be found in
\cite{De_Illinois_2011, Lacey, Muckenhoupt,Rubin_Math.Notes_1983, Sawyer_weak, Sawyer_two weight}.
On the other hand, as mentioned in Stein-Wainger \cite{Stein_Wainger}, the discrete analogue of the fractional integral operator is given by
\begin{equation*}
(\mathcal {I}_{\lambda}f)(k)=\sum_{j\in \mathbb{Z}^n, j\neq k}\frac{f(j)}{|k-j|^{n-\lambda}}.
\end{equation*}

As the second application of our main results on Young-type inequalities,
we will study the fractional integral operator in both discrete and
continuous case. Our results and methods allow us to obtain the sharp
conditions for the boundedeness of fractional integral operator on weighted
Lebesgue spaces. Especially, in the continuous case, we optimize some
previous results by finding the sharp conditions for the boundedness of $
L^{p}-L^{q}$ estimates of fractional integral operators with power weights.
Our proof mainly depends on the discretization of the operator, which is
quite different from the methods used by other authors.

Now, we list our main results associated with fractional integral operators.
\begin{theorem}[Fractional integral operator, discrete form, weight $\langle k\rangle$] \label{Sharpness of fractional integration operator, discrete form}
Suppose $1\leq p, q\leq \infty$, $t, s\in \mathbb{R}$. Then
\begin{equation*}
\mathcal {I}_{\lambda}: l(p,t)\rightarrow l(q,s)
\end{equation*}
if and only if $(q, p, s, t)$ satisfies one of the following conditions
\begin{eqnarray}
&&(\mathcal {C}_1)\begin{cases}
s\leq t,\\
\frac{\lambda}{n}+\Big(\frac{1}{q}+\frac{s}{n}\Big)\vee 0<\Big(\frac{1}{p}+\frac{t}{n}\Big)\vee 0,\\
\frac{\lambda}{n}+\frac{1}{q}+\frac{s}{n}\leq 1,\\
(q', -s)=(1, \lambda-n)~\text{if}~ \frac{\lambda}{n}+\frac{1}{q}+\frac{s}{n}=1;
\end{cases}
\\
&&(\mathcal {C}_3)~\begin{cases}
s\leq t,\\
p=1, t=\lambda-n, \frac{1}{q}+\frac{s}{n}<0;\\
\end{cases}
\\
&&(\mathcal {C}_4)~\begin{cases}
s\leq t,\\
\frac{\lambda}{n}+\frac{1}{q}+\frac{s}{n}=\frac{1}{p}+\frac{t}{n}, \frac{1}{q}\leq \frac{1}{p},\\
\frac{1}{q}+\frac{s}{n}>0,~\frac{1}{p}+\frac{t}{n}<1,\\
q\neq \infty,~p\neq 1,~if~s=t.
\end{cases}
\end{eqnarray}
\end{theorem}

\begin{theorem}[Fractional integral operator, continuous form, weight $|x|$] \label{Sharpness of fractional integration operator, power weight}
Let $1\leq p, q\leq \infty$, $t, s\in \mathbb{R}$. Then
\begin{equation*}
I_{\lambda}: \mathbb{L}(p,t)\rightarrow \mathbb{L}(q,s)
\end{equation*}
if and only if $(q, p, s, t)$ satisfies $\mathcal {C}_4$.
\end{theorem}

In Theorem \ref{Sharpness of fractional integration operator, discrete form},
to maintain the unified format of proof as in Theorem \ref{Sharpness of weighted Young's inequality, discrete form},
we skip $\mathcal {C}_2$ to name the conditions.
In fact, in our unified method,
the conditions $\mathcal {C}_i$ and $\mathcal {A}_i$ are correspondence for each i=1,2,3,4.
The subscript $i=2$ means the double endpoint cases which can be proved to be trivial.
Under this method of classification, we actually have $\mathcal {C}_2=\emptyset$.
One can also see the proof of Theorem \ref{Sharpness of fractional integration operator, discrete form} in this direction.

In addition, the subscript $i=1$ means the case in which we can use embedding argument to reduce the proof to
a more standard case (see Proposition \ref{For reduction, discrete form}),
the case $i=3$ is actually the dual of the case $i=1$.
Finally, condition for $i=4$ collect the cases which is closely linked to the continuous form.

We remark that each of Theorems 1.1 to 1.7 can be verified independently.
For convenience, we sometimes use one Theorem \ref{Sharpness of weighted Young's inequality, discrete form} to prove other theorems for an
easy approach. Our methods in this paper is in the spirit of discretization,
even in the process of dealing with the continuous case. This is totally
different from the methods used in the other references about this topic.

We also remark that the nonnegative functions (or sequences) are enough for most of the proofs in this paper.
So, if there is no special explanation, the functions we use in the proofs should be presumed nonnegative.
\section{Preliminaries and Definitions}
Let $C$ be a positive constant that may depend on $n,\,p_{i},\,q_{i},\,s_{i},\,t,\,\lambda\,(i=1,\,2).$ The notation $X\lesssim Y$ denotes the statement that $X\leq CY$,
the notation $X\sim Y$ means the statement $X\lesssim Y\lesssim X$, and the
notation $X\simeq Y$ denotes the statement $X=CY$. For a multi-index $
k=(k_{1},k_{2},...,k_{n})\in \mathbb{Z}^{n}$, we denote $|k|_{\infty }:=
\sup_{i=1,2,...,n}|k_{i}|$, and $\langle k\rangle :=(1+|k|^{2})^{{
1}/{2}}.$

Let $\mathscr{S}:= \mathscr{S}(\mathbb{R}^{n})$ be the Schwartz space
and $\mathscr{S}':=\mathscr{S}'(\mathbb{R}^{n})$ be the space of tempered distributions.
We define the Fourier transform $\mathcal {F}f$ and the inverse Fourier $\mathcal {F}^{-1}f$ of $f\in \mathscr{S}(\mathbb{R}^{n})$ by
$$
\mathcal {F}f(\xi)=\hat{f}(\xi)=\int_{\mathbb{R}^{n}}f(x)e^{-2\pi ix\cdot \xi}dx
,
~~
\mathcal {F}^{-1}f(x)=\hat{f}(-x)=\int_{\mathbb{R}^{n}}f(\xi)e^{2\pi ix\cdot \xi}d\xi.
$$

We recall, in the following, the definitions and some properties of the
function spaces involved in this paper.

\begin{definition}
Let $s\in \mathbb{R}, 0<p,q\leq \infty$. The function space $\mathbb{L}_x(p,s)$ consists of all measurable functions $f$ such that
\begin{numcases}{\|f\|_{\mathbb{L}_x(p,s)}=}
     \left(\int_{\mathbb{R}^n}|f(x)|^p |x|^{ps} dx\right)^{1/p}, &$p<\infty$  \nonumber \\
     ess\sup_{x\in \mathbb{R}^n}|f(x)||x|^s,  &$p=\infty$ \nonumber
\end{numcases}
is finite. The function space $L_x(p,s)$ consists of measurable functions $f$ such that
\begin{numcases}{\|f\|_{L_x(p,s)}=}
     \left(\int_{\mathbb{R}^n}|f(x)|^p \langle x\rangle^{ps} dx\right)^{1/p}, &$p<\infty$  \nonumber \\
     ess\sup_{x\in \mathbb{R}^n}|f(x)| \langle x\rangle^s, &$p=\infty$ \nonumber
\end{numcases}
is finite.
If $f$ is defined on $\mathbb{Z}^n$, we denote its $l_{k}(p,s)$ norm
\begin{numcases}{\|f\|_{l_k(p,s)}=}
\left(\sum_{k\in \mathbb{Z}^n}|f(k)|^p \langle k\rangle^{ps}\right)^{1/p}, &$p<\infty$ \nonumber\\
\sup_{k\in \mathbb{Z}^n}|f(k)|\langle k\rangle^s, \hspace{15mm}&$p=\infty$ \nonumber
\end{numcases}
and let $l_k(p,s)$ be the (quasi-)Banach space of functions $f: \mathbb{Z}^n\rightarrow \mathbb{C}$ whose $l_k(p,s)$ norm is finite.
We write $\mathbb{L}(p,s)$, $L(p, s)$, $l(p,s)$ for short respectively, if there is no confusion.
We also  denote $L^p=\mathbb{L}(p,0)=L(p,0)$, $l^p=l(p,0)$ for short.
\end{definition}

To introduce the modulation space,  we first give the definition of the short-time Fourier transform.
For a fixed nonzero $\phi\in \mathscr{S}$, the short-time Fourier
transform of $f\in \mathscr{S}$ with respect to the window function $\phi$ is given by
\begin{equation*}
V_{\phi}f(x,\xi)=\int_{\mathbb{R}^n}f(y)\overline{\phi(y-x)}e^{-2\pi iy\cdot \xi}dy.
\end{equation*}
The norm on modulation space is given by
\begin{equation}\label{modulation space def. with window function}
\begin{split}
\|f\|_{M_{p,q}^{s}}&=\big\|\|V_{\phi}f(x,\xi)\|_{L^p_x}\big\|_{L_\xi(q,s)}
\\&
=\left(\int_{\mathbb{R}^n}\left(\int_{\mathbb{R}^n}|V_{\phi}f(x,\xi)|^{p}dx\right)^{q/p}\langle \xi\rangle^{sq}d\xi\right)^{{1}/{q}},
\end{split}
\end{equation}
with a natural modification for $p=\infty$ or $q=\infty$.
Note that this definition is independent of the choice of the window function.

Applying the frequency-uniform localization techniques, one can give an
alternative definition of modulation spaces (see \cite{WH_JDE_2007} for
details).
For $k\in \mathbb{Z}^{n},$ we denote by $Q_{k}$ the unit cube
centered at $k$. The family $\{Q_{k}\}_{k\in \mathbb{Z}^{n}}$ constitutes a
decomposition of $\mathbb{R}^{n}$. Let $\rho \in \mathscr {S}(\mathbb{R}
^{n}),$ $\rho :\mathbb{R}^{n}\rightarrow \lbrack 0,1]$ be a smooth function
satisfying $\rho (\xi )=1$ for $|\xi |_{\infty }\leq {1}/{2}$ and $\rho
(\xi )=0$ for $|\xi |\geq 3/4$. Let $\rho _{k}$ be a translation of $\rho $,
\begin{equation*}
\rho _{k}(\xi )=\rho (\xi -k),\text{ }k\in \mathbb{Z}^{n}.
\end{equation*}
Since $\rho _{k}(\xi )=1$ for $\xi \in Q_{k}$, \ we have that $\sum_{k\in
\mathbb{Z}^{n}}\rho _{k}(\xi )\geq 1$ for all $\xi \in \mathbb{R}^{n}$.
Denote
\begin{equation*}
\sigma _{k}(\xi )=\rho _{k}(\xi )\left( \sum_{l\in \mathbb{Z}^{n}}\rho
_{l}(\xi )\right) ^{-1},~~~~k\in \mathbb{Z}^{n}.
\end{equation*}
Then $\{\sigma _{k}(\xi )\}_{k\in \mathbb{Z}^{n}}$ constitutes a smooth
decomposition of $\mathbb{R}^{n},$ where $\sigma _{k}(\xi )=\sigma (\xi -k)$. The frequency-uniform decomposition operators are defined by
\begin{equation*}
\Box _{k}:=\mathscr{F}^{-1}\sigma _{k}\mathscr{F}
\end{equation*}
for $k\in \mathbb{Z}^{n}$. With the family \ $\left\{ \Box _{k}\right\}
_{k\in \mathbb{Z}^{n}}$,\ an alternative norm of modulation space can be
defined by
\begin{equation*}
\left( \sum_{k\in \mathbb{Z}^{n}}\langle k\rangle ^{sq}\Vert \Box _{k}f\Vert
_{p}^{q}\right) ^{{1}/{q}},
\end{equation*}%
with a natural modification for $p=\infty $ or $q=\infty $. We recall that
this definition is independent of the choice of $\sigma $ and that this norm
is equivalent to the norm defined in (\ref{modulation space def. with window
function}) (see \cite{WH_JDE_2007}). So we use the same symbol $\Vert f\Vert
_{M_{p,q}^{s}}$ to denote these two modulation space norms.

\begin{lemma}[Embedding of $L^{p}$ with
Fourier compact support, \cite{Tribel_83}]
\label{embedding of Lp with Fourier compact support} Let $0<p_{1}\leq
p_{2}\leq \infty $ and assume $supp{\hat{f}}\subseteq B(x_0,R)$ for some $x_0\in \mathbb{R}^n$, $R>0$. We have
\begin{equation*}
\Vert f\Vert _{L^{p_{2}}}\leq CR^{n(\frac{1}{p_{1}}-\frac{1}{p_{2}})}\Vert
f\Vert _{L^{p_{1}}},
\end{equation*}%
where $C$ is independent of $f$ and $x_0$.
\end{lemma}

Next, we list some propositions used in the proof of our main theorems.
These propositions are not difficult to be verified, so we only give partial proof details and give some hints.

\begin{proposition}[Sharpness of embedding, discrete form] \label{Sharpness of embedding, discrete form}
Suppose $0<q,q_1,q_2\leq \infty$, $s,s_1,s_2\in \mathbb{R}$. Then
\begin{equation*}
l(q_1,s_1)\subset l(q_2,s_2)
\end{equation*}
holds if and only if
\begin{equation*}
\begin{cases}
s_2\leq s_1  \\
\frac{1}{q_2}+\frac{s_2}{n}< \frac{1}{q_1}+\frac{s_1}{n}
\end{cases}
\text{or} \hspace{10mm}
\begin{cases}
s_2=s_1\\
q_2=q_1.
\end{cases}
\end{equation*}
\end{proposition}
\begin{proof}
The sufficiency can be verified by the H\"{o}lder inequality and the fact that $l_{q_1}\subset l_{q_2}$ for $1/q_2\leqslant 1/q_1$.

To prove the necessity, we firstly obtain $\frac{1}{q_2}+\frac{s_2}{n}\leqslant\frac{1}{q_1}+\frac{s_1}{n}$
by the same method as in the proof of (\ref{for proof, preliminary, 1}).
For a fixed $N\in \mathbb{N}$, we take $a_{k,N}=1$ if $k=N$ and $a_{k,N}=1$ if $k=N$, then $s_2\leqslant s_1$ follows by
$$\langle N\rangle^{s_2}\sim \|\{a_{k,N}\}\|_{l(q_2,s_2)}\lesssim \|\{a_{k,N}\}\|_{l(q_1,s_1)}\sim \langle N\rangle^{s_1}$$
as $|N|\rightarrow \infty$.

Especially, for $\frac{1}{q_2}+\frac{s_2}{n}=\frac{1}{q_1}+\frac{s_1}{n}$, we take
\begin{equation*}
b_{k,N}=
\begin{cases}
\langle k\rangle^{-n(1/q_1+s_1/n)}, & \text{for}~|k|\leq N, \\
0, & \text{otherwise}.
\end{cases}
\end{equation*}
Observing that
\begin{equation*}
  (\ln N)^{1/q_2}\sim \|\{b_{k,N}\}\|_{l(q_2,s_2)}\lesssim \|\{b_{k,N}\}\|_{l(q_1,s_1)}\sim (\ln N)^{1/q_1},
\end{equation*}
we deduce $1/q_2\leqslant 1/q_1$ by letting $|N|\rightarrow \infty$.
Recalling $\frac{1}{q_2}+\frac{s_2}{n}=\frac{1}{q_1}+\frac{s_1}{n}$ and $s_2\leqslant s_1$, we actually have $s_2=s_1$ and $q_2=q_1$ in this case.
\end{proof}

In this paper, the space $l(q,s)$ is quite important in our proof.

\begin{proposition}[Sharpness of embedding, continuous form] \label{Sharpness of embedding, continuous form}
Suppose $0<q,q_1,q_2\leq \infty$, $s,s_1,s_2\in \mathbb{R}$. Then
\begin{equation*}
L(q_1,s_1)\subset L(q_2,s_2)
\end{equation*}
holds if and only if
\begin{equation*}
\begin{cases}
s_2\leq s_1  \\
\frac{1}{q_2}\geq \frac{1}{q_1}\\
\frac{1}{q_2}+\frac{s_2}{n}< \frac{1}{q_1}+\frac{s_1}{n}
\end{cases}
\text{or} \hspace{10mm}
\begin{cases}
s_2=s_1\\
q_2=q_1.
\end{cases}
\end{equation*}
\end{proposition}
\begin{proof}
  The sufficiency can be verified by H\"{o}lder's inequality.
  In the necessity part, we take $f(x)=\chi_{B(0,a)}$ for $a\in (0,1)$.
  Then $1/q_2\geqslant 1/q_1$ follows by letting $a\rightarrow 0$ in
  \begin{equation*}
    a^{n/q_2}\sim \|f\|_{L(q_2,s_2)}\lesssim \|f\|_{L(q_1,s_1)}\sim a^{n/q_1}.
  \end{equation*}
  The rest of the proof is similarly as that in Proposition \ref{Sharpness of embedding, discrete form}.
\end{proof}

\begin{proposition}[Young's inequality, discrete form] \label{Sharpness of the Young's inequality, discrete form}
Suppose $0<q,q_1,q_2\leq \infty$. Then
\begin{equation*}
l^{q_1}\ast l^{q_2} \subset l^{q}
\end{equation*}
holds if and only if
\begin{equation*}
\begin{cases}
1+\frac{1}{q}\leq\frac{1}{q_1}+\frac{1}{q_2}
\\
\frac{1}{q}\leq \frac{1}{q_1},\frac{1}{q}\leq \frac{1}{q_2}.
\end{cases}
\end{equation*}
\end{proposition}
\begin{proof}
Since the necessity part
can be verified by the same method as in the proof of Theorem \ref{Sharpness of weighted Young's inequality, discrete form},
we only give the proof for the sufficiency.
If $q\leqslant 1$, we have $l^q\ast l^q\subset l^q$,
then the sufficiency can be verified by
  \begin{equation*}
    l^{q_1}\ast l^{q_2} \subset l^{q}\ast l^{q} \subset l^{q},
  \end{equation*}
  where we use $l^{q_1}\subset l^q$, $l^{q_2}\subset l^q$ by the fact $1/q\leqslant 1/q_1,1/q_2$.

If $q\geqslant 1$, $q_1\geqslant 1$, $q_2\leqslant 1$, we deduce
\begin{equation*}
  l^{q_1}\ast l^{q_2} \subset l^{q}\ast l^{1} \subset l^{q},
\end{equation*}
where we use $l^{q_2}\subset l^1$ in this case.
By the symmetry, the case $q\geqslant 1$, $q_2\geqslant 1$, $q_1\leqslant 1$ can be handled by the same way.

If $q\geqslant 1$, $q_1\geqslant 1$, $q_2\geqslant 1$, we can choose $r_j\geqslant 1 (j=1,2)$ such that
\begin{equation*}
  1+1/q=1/r_1+1/r_2,\ r_j\geqslant q_j(j=1,2).
\end{equation*}
Using Young's inequality, we obtain $l^{r_1}\ast l^{r_2} \subset l^{q}$.
It implies that
\begin{equation*}
  l^{q_1}\ast l^{q_2} \subset l^{r_1}\ast l^{r_2} \subset l^{q}
\end{equation*}
by using $l^{q_i}\subset l^{r_i}(i=1,2)$ in this case.
\end{proof}

\begin{proposition}[Young's inequality, continuous form] \label{Sharpness of the Young's inequality, continuous form}
Let $0<q,q_1,q_2 \leq \infty$. Then
\begin{equation*}
L^{q_1}\ast L^{q_2}\subset L^q
\end{equation*}
holds if and only if
\begin{equation*}
\begin{cases}
1+\frac{1}{q}= \frac{1}{q_1}+\frac{1}{q_2}
\\
\frac{1}{q}\leq \frac{1}{q_1},\frac{1}{q}\leq \frac{1}{q_2}.
\end{cases}
\end{equation*}
\end{proposition}
\begin{proof}
  The necessity can be verified by the same method as in the proof of Theorem \ref{Sharpness of weighted Young's inequality, power weight}.
  Observing the conditions is equivalent to $1/q=1/q_1+1/q_2$, $q,q_1,q_2\geqslant 1$,
  the sufficiency follows by the classical Young's inequality.
\end{proof}

\begin{proposition}[Integral capability of weight $\langle k\rangle$] \label{Integral capability of weight, discrete form}
Suppose $s>0$, $s\leq s_1$, $s\leq s_2$. Then
\begin{equation*}
\begin{split}
&\left\{t\in (0,\infty]: \left\|\frac{\langle k\rangle^s}{\langle k-j\rangle^{s_1}\langle j\rangle^{s_2}} \right\|_{l_j^t}\lesssim 1,~for~all~k\in \mathbb{Z}^n \right\}
\\
&\qquad\qquad=
\begin{cases}
(\frac{n}{s_1+s_2-s},\infty],~\text{if}~s=s_1~or~s=s_2,
\\
[\frac{n}{s_1+s_2-s},\infty],~\text{if}~s<s_1~and~s<s_2.
\end{cases}
\end{split}
\end{equation*}
\end{proposition}
\begin{proof}
By a direct calculation,
\begin{equation*}
\begin{split}
&\sum_{j\in \mathbb{Z}^n}\frac{\langle k\rangle^{st}}{\langle
k-j\rangle^{s_1t}\langle j\rangle^{s_2t}} \\
= & \sum_{|j|\leq {|k|}/{2}}+\sum_{|j-k|\leq {|k|}/{2}
}+\sum_{|j|\geq 2k}+\sum_{\text{others}} \\
\sim & \sum_{|j|\leq {|k|}/{2}}\frac{\langle k\rangle^{(s-s_1)t}}{
\langle j\rangle^{s_2t}} + \sum_{|j|\leq {|k|}/{2}}\frac{\langle
k\rangle^{(s-s_2)t}}{\langle j\rangle^{s_1t}} + \sum_{|j|\geq 2|k|}\frac{
\langle k\rangle^{st}}{\langle j\rangle^{(s_1+s_2)t}} + \langle
k\rangle^{n+(s-s_1-s_2)t} \\
=:&I+I\!I+I\!I\!I+I\!V.
\end{split}%
\end{equation*}
One can verify that $I\!I\!I$ and $I\!V$ have uniform bounds on $k$ if and
only if $t\geq {n}/{(s_1+s_2-s)}$. On the other hand, under the condition
$t\geq {n}/{(s_1+s_2-s)}$, the term $I$ also has uniform bounds on $k$
, unless $t={n}/{s_2}$ and $s=s_1$. Similarly, one can verify that $I\!I$
also has uniform bounds on $k$ unless $t={n}/{s_1}$ and $s=s_2$.
\end{proof}

The following two propositions can be verified by the similar technique as in the proof of Proposition \ref{Integral capability of weight, discrete form}, so we omit their proofs.
\begin{proposition}[Integral capability of weight $\langle x\rangle$] \label{Integral capability of weight, continuous form}
Suppose $s>0$, $s\leq s_1$, $s\leq s_2$. Then
\begin{equation*}
\begin{split}
&\left\{t: \left\|\frac{\langle x\rangle^s}{\langle x-y\rangle^{s_1}\langle y\rangle^{s_2}} \right\|_{L_y^t}
\lesssim 1,~for~all~x\in \mathbb{R}^n \right\}
\\
&\qquad\qquad=
\begin{cases}
(\frac{n}{s_1+s_2-s},\infty],~\text{if}~s=s_1~\text{or}~s=s_2,
\\
[\frac{n}{s_1+s_2-s},\infty],~\text{if}~s<s_1~\text{and}~s<s_2.
\end{cases}
\end{split}
\end{equation*}
\end{proposition}

\begin{proposition}[Integral capability of weight $|x|$] \label{Integral capability of weight, power weight}
Suppose $s>0$, $s< s_1$, $s< s_2$. Then
\begin{equation*}
\left\|\frac{|x|^s}{|x-y|^{s_1}|y|^{s_2}} \right\|_{L_y^{\frac{n}{s_1+s_2-s}}}\lesssim 1
\end{equation*}
for all $x\in \mathbb{R}^n $.
\end{proposition}

\section{Discrete weighted Young's inequality---Proof of Theorem \ref{Sharpness of weighted Young's inequality, discrete form}}

\subsection{Notations and procedure of the proof}$~$\\

We start this section by defining the set
\begin{equation*}
A=\left\{ (\mathbf{q},\mathbf{s})\in \lbrack 1,\infty]^{3}\times \mathbb{R}%
^{3}:~l(q_{1},s_{1})\ast l(q_{2},s_{2})\subset l(q,s)\right\} .
\end{equation*}%
We now describe the strategy to characterize the set \ $A.$

Use $A_{i}$ to denote the set of all $(\mathbf{q},\mathbf{s})\in \lbrack
1,\infty]^{3}\times \mathbb{R}^{3}$ satisfying condition $\mathcal{A}_{i}$,
respectively, for $i=1,\,2,\,3,\,4$, $\ X$ to denote the set of all $(\mathbf{q},%
\mathbf{s})\in \lbrack 1,\infty]^{3}\times \mathbb{R}^{3}$ satisfying
\begin{equation*}
\begin{cases}
s\leq s_1,~s\leq s_2,~0\leq s_1+s_2;\\
1+\Big(\frac{1}{q}+\frac{s}{n}\Big)\vee 0\leq\Big(\frac{1}{q_1}+\frac{s_1}{n}\Big)\vee 0+\Big(\frac{1}{q_2}+\frac{s_2}{n}\Big)\vee 0;\\
\frac{1}{q}+\frac{s}{n}\leq\frac{1}{q_1}+\frac{s_1}{n},~\frac{1}{q}+\frac{s}{n}\leq\frac{1}{q_2}+\frac{s_2}{n},~
1\leq \frac{1}{q_1}+\frac{s_1}{n}+\frac{1}{q_2}+\frac{s_2}{n};\\
(q,s)=(q_1,s_1), ~\text{if}~ \frac{1}{q}+\frac{s}{n}=\frac{1}{q_1}+\frac{s_1}{n}; \\
(q,s)=(q_2,s_2), ~\text{if}~ \frac{1}{q}+\frac{s}{n}=\frac{1}{q_2}+\frac{s_2}{n};\\
(q'_1, -s_1)=(q_2,s_2),~\text{if}~ 1= \frac{1}{q_1}+\frac{s_1}{n}+\frac{1}{q_2}+\frac{s_2}{n}.
\end{cases}
\end{equation*}
We also use $X_{1}$ to denote the set of all pairs $(\mathbf{q},\mathbf{s}%
)\in X$ satisfying
\begin{equation*}
1+\Big(\frac{1}{q}+\frac{s}{n}\Big)\vee 0<\Big(\frac{1}{q_{1}}+\frac{s_{1}}{n%
}\Big)\vee 0+\Big(\frac{1}{q_{2}}+\frac{s_{2}}{n}\Big)\vee 0.
\end{equation*}%
By this notation, one easily checks that $X_{1}=A_{1}$.
We use $X_{2}$ to denote the set of all $(\mathbf{q},\mathbf{s})\in X\backslash X_1$
satisfying
\begin{equation*}
\frac{1}{q}+\frac{s}{n}=\frac{1}{q_1}+\frac{s_1}{n}~\text{or}~\frac{1}{q}+\frac{s}{n}=\frac{1}{q_2}+\frac{s_2}{n}
~\text{or}~\frac{1}{q}+\frac{s}{n}=0.
\end{equation*}
We use $X_{3}$ to denote the set of all $(\mathbf{q},\mathbf{s})\in X\backslash (X_1\cup X_2)$
satisfying
\begin{equation*}
\frac{1}{q}+\frac{s}{n}<0.
\end{equation*}
Use $X_{4}$ to denote the set of all $(\mathbf{q},\mathbf{s})\in X\backslash (X_1\cup X_2)$
satisfying
\begin{equation*}
\frac{1}{q}+\frac{s}{n}>0.
\end{equation*}
By the above definition, we also obtain that $X_3$ is the set of all
$(\mathbf{q},\mathbf{s})\in \lbrack 1,\infty]^{3}\times \mathbb{R}^{3}$ satisfying
\begin{equation*}
\begin{cases}
s\leq s_1,~s\leq s_2,~0\leq s_1+s_2,\\
1+\Big(\frac{1}{q}+\frac{s}{n}\Big)\vee 0=\Big(\frac{1}{q_1}+\frac{s_1}{n}\Big)\vee 0+\Big(\frac{1}{q_2}+\frac{s_2}{n}\Big)\vee 0,\\
\frac{1}{q}+\frac{s}{n}<\frac{1}{q_1}+\frac{s_1}{n},~\frac{1}{q}+\frac{s}{n}<\frac{1}{q_2}+\frac{s_2}{n},
~1\leq \frac{1}{q_1}+\frac{s_1}{n}+\frac{1}{q_2}+\frac{s_2}{n},
~\frac{1}{q}+\frac{s}{n}<0,~\\
(q'_1, -s_1)=(q_2,s_2)~\text{if}~ 1= \frac{1}{q_1}+\frac{s_1}{n}+\frac{1}{q_2}+\frac{s_2}{n};
\end{cases}
\end{equation*}
$X_{4}$ is the set of all
$(\mathbf{q},\mathbf{s})\in \lbrack 1,\infty]^{3}\times \mathbb{R}^{3}$ satisfying
\begin{equation*}
\begin{cases}
s\leq s_{1},~s\leq s_{2},~0\leq s_{1}+s_{2}, \\
1+\frac{1}{q}+\frac{s}{n}=\frac{1}{q_{1}}+\frac{s_{1}}{n}+\frac{1}{q_{2}}+
\frac{s_{2}}{n}, \\
\frac{1}{q}+\frac{s}{n}<\frac{1}{q_{1}}+\frac{s_{1}}{n},~\frac{1}{q}+\frac{s
}{n}<\frac{1}{q_{2}}+\frac{s_{2}}{n},~\frac{1}{q}+\frac{s}{n}>0.
\end{cases}
\end{equation*}
The set $X$ \ now is the union of mutually disjoint
sets $X_{j}:$%
\begin{equation*}
X=\bigcup_{j=1}^{4}X_{j},
\end{equation*}%
where $\ X_{i}\cap X_{j}=\emptyset $ whenever $i\neq j$.

To prove Theorem 1.1, it suffices to show $A\subset X$ and $A\cap
X_{j}=A_{j} $. Then the conclusion of Theorem 1.1 follows from the easy fact
\begin{equation*}
A=A\cap X=A\cap \left( \bigcup_{i=1}^{4}X_{i}\right)
=\bigcup_{i=1}^{4}(A\cap X_{i})=\bigcup_{i=1}^{4}A_{i}.
\end{equation*}

\subsection{The proof of $A\subset X$.}$~$
\\
We define two sequences \ $\{a_{k,N}\}_{k\in \mathbb{Z}^{n}}$ \ and \ $\{b_{k,N}\}_{k\in \mathbb{Z}^{k}}$ \ for each natural number \ $N,$ where
\begin{equation*}
a_{k,N}=
\begin{cases}
1, & \text{for}~|k|\leq N, \\
0, & \text{otherwise},
\end{cases}%
\hspace{10mm}b_{k,N}=
\begin{cases}
1, & \text{for}~|k|\leq 2N, \\
0, & \text{otherwise}.
\end{cases}
\end{equation*}
For $k\in \mathbb{Z}^{n},~|k|\leq N$, we have
\begin{equation*}
\sum_{j\in \mathbb{Z}^{n}}a_{k-j,N}b_{j,N}\geq \sum_{|j|\leq
2N}a_{k-j,N}\geq \sum_{|j|\leq N}a_{j,N}\geq N^{n},
\end{equation*}
and
\begin{equation*}
\Vert \{a_{k,N}\}\ast\{b_{k,N}\}\Vert _{l(q,s)}\geq N^{n}\Big(\sum_{|k|\leq
N}\langle k\rangle ^{sq}\Big)^{{1}/{q}}.
\end{equation*}
So, by the definition of \ $A$ \ we deduce
\begin{equation}\label{for proof, 11}
N^{n}\Big(\sum_{|k|\leq N}\langle k\rangle ^{sq}\Big)^{{1}/{q}}\lesssim
\Big(\sum_{|k|\leq N}\langle k\rangle ^{s_{1}q_{1}}\Big)^{{1}/{q_{1}}}
\Big(\sum_{|k|\leq 2N}\langle k\rangle ^{s_{2}q_{2}}\Big)^{{1}/{q_{2}}}.
\end{equation}
For any $\delta >0$, we have
\begin{equation*}
N^{n[({1}/{q}+{s}/{n})\vee 0]-\delta }\lesssim \left( \sum_{|k|\leq
N}\langle k\rangle ^{sq}\right) ^{{1}/{q}}\lesssim N^{n[({1}/{q}+
{s}/{n})\vee 0]+\delta }.
\end{equation*}
Combining it with (\ref{for proof, 11}),
we then have that
\begin{equation*}
N^{n}\cdot N^{n[({1}/{q}+{s}/{n})\vee 0]-\delta }\lesssim
N^{n[({1}/{q_{1}}+{s_{1}}/{n})\vee 0]+\delta }\cdot N^{n[({1}/{
q_{2}}+{s_{2}}/{n})\vee 0]+\delta }
\end{equation*}
as $N\rightarrow \infty $, which implies
\begin{equation*}
n+n[(\frac{1}{q}+\frac{s}{n})\vee 0]-\delta \leq n[(\frac{1}{q_{1}}+\frac{
s_{1}}{n})\vee 0]+\delta +n[(\frac{1}{q_{2}}+\frac{s_{2}}{n})\vee 0]+\delta .
\end{equation*}
Letting $\delta \rightarrow 0$, we further obtain
\begin{equation}\label{for proof, preliminary, 1}
1+\Big(\frac{1}{q}+\frac{s}{n}\Big)\vee 0\leq \Big(\frac{1}{q_{1}}+\frac{
s_{1}}{n}\Big)\vee 0+\Big(\frac{1}{q_{2}}+\frac{s_{2}}{n}\Big)\vee 0.
\end{equation}

On the other hand, we may take $b_{0}=1$ and $b_{k}=0\ (k\neq 0)$ in the
inequality
\begin{equation*}
\Vert \{a_{k}\}\ast \{b_{k}\}\Vert _{l(q,s)}\lesssim \Vert \{a_{k}\}\Vert
_{l(q_{1},s_{1})}\Vert \{b_{k}\}\Vert _{l(q_{2},s_{2})}
\end{equation*}
to deduce
\begin{equation*}
\Vert \{a_{k}\}\Vert _{l(q,s)}\lesssim \Vert \{a_{k}\}\Vert
_{l(q_{1},s_{1})},
\end{equation*}
which implies that
\begin{equation*}
l(q_{1},s_{1})\subset l(q,s).
\end{equation*}
A similar argument then gives the inclusion
\begin{equation*}
l(q_{2},s_{2})\subset l(q,s).
\end{equation*}
Using an easy dual argument, we have
\begin{equation}\label{for proof, 12}
\Vert \{a_{k}\}\ast \{b_{k}\}\Vert _{l(q_{1}^{\prime },-s_{1})}\lesssim
\Vert \{a_{k}\}\Vert _{l(q^{\prime },-s)}\Vert \{b_{k}\}\Vert
_{l(q_{2},s_{2})}.
\end{equation}
In fact, by the assumption $l(q_1,s_1)\ast l(q_2,s_2)\subset l(q,s)$, we obtain
\begin{equation*}
  \begin{split}
    |\sum_{k\in \mathbb{Z}^n}\sum_{j\in \mathbb{Z}^n}c_{k-j}b_ja_k|
    \leqslant &
    \|\{c_k\}\ast\{b_k\}\|_{l(q,s)}\|\{a_k\}\|_{l(q',-s)}
    \\
    \leqslant &
    \|\{c_k\}\|_{l(q_1,s_1)}\|\{b_k\}\|_{l(q_2,s_2)}\|\{a_k\}\|_{l(q',-s)},
  \end{split}
\end{equation*}
where we use the H\"{o}lder inequality in the first inequality.
Observing that
\begin{equation*}
  \sum_{k\in \mathbb{Z}^n}\sum_{j\in \mathbb{Z}^n}c_{k-j}b_ja_k
  =
  \sum_{k\in \mathbb{Z}^n}\sum_{j\in \mathbb{Z}^n}c_jb_{k-j}a_k
  =
\sum_{j\in \mathbb{Z}^n}c_j\sum_{k\in \mathbb{Z}^n}b_{k-j}a_k,
\end{equation*}
we actually have
\begin{equation*}
  \begin{split}
    |\sum_{j\in \mathbb{Z}^n}c_j\sum_{k\in \mathbb{Z}^n}b_{k-j}a_k|
    \leqslant
    \|\{c_k\}\|_{l(q_1,s_1)}\|\{b_k\}\|_{l(q_2,s_2)}\|\{a_k\}\|_{l(q',-s)}.
  \end{split}
\end{equation*}
By the arbitrary of $\{c_k\}_{k\in \mathbb{Z}^n}$, we obtain
\begin{equation*}
  \begin{split}
    \|\{\sum_{k\in \mathbb{Z}^n}b_{k-j}a_k\}\|_{l_j(q_1',-s_1)}
    \leqslant
    \|\{b_k\}\|_{l(q_2,s_2)}\|\{a_k\}\|_{l(q',-s)}.
  \end{split}
\end{equation*}
Rewriting $b_{j}=\tilde{b}_{-j}$, we obtain $\sum_{k\in \mathbb{Z}^n}b_{k-j}a_k=(\{a_k\}\ast\{b_k\})(j)$.
Then
\begin{equation*}
  \begin{split}
    \|\{a_k\}\ast\{\tilde{b}_k\}\|_{l(q_1',-s_1)}
    \leqslant &
    \|\{b_k\}\|_{l(q_2,s_2)}\|\{a_k\}\|_{l(q',-s)}
    \\
    = &
    \|\{a_k\}\|_{l(q',-s)}\|\{\tilde{b}_k\}\|_{l(q_2,s_2)},
  \end{split}
\end{equation*}
which is just the inequality (\ref{for proof, 12}).

Proceed the argument as above, we obtain that
\begin{equation*}
l(q_{2},s_{2})\subset l(q_{1}^{\prime },-s_{1}).
\end{equation*}
Also, invoking Proposition \ref{Sharpness of embedding, discrete form}, we can get
\begin{equation*}
\frac{1}{q}+\frac{s}{n}\leq \frac{1}{q_{1}}+\frac{s_{1}}{n}~,~\frac{1}{q}+%
\frac{s}{n}\leq \frac{1}{q_{2}}+\frac{s_{2}}{n}~,~\frac{1}{q_{1}^{\prime }}+%
\frac{-s_{1}}{n}\leq \frac{1}{q_{2}}+\frac{s_{2}}{n},
\end{equation*}
and
\begin{equation*}
s\leq s_{1},~s\leq s_{2},~-s_{1}\leq s_{2},
\end{equation*}
where
\begin{equation}\label{for proof, 13}
\begin{cases}
(q,s)=(q_{1},s_{1}),~\text{if}~\frac{1}{q}+\frac{s}{n}=\frac{1}{q_{1}}+\frac{
s_{1}}{n}; \\
(q,s)=(q_{2},s_{2}),~\text{if}~\frac{1}{q}+\frac{s}{n}=\frac{1}{q_{2}}+\frac{
s_{2}}{n}; \\
(q_{1}^{\prime },-s_{1})=(q_{2},s_{2}),~\text{if}~\frac{1}{q_{1}^{\prime }}+
\frac{-s_{1}}{n}=\frac{1}{q_{2}}+\frac{s_{2}}{n}.
\end{cases}
\end{equation}
We emphasize that, in our proofs in this paper,
many endpoint cases will be reduced to an embedding relations satisfying the condition such as the if parts in
(\ref{for proof, 13}),
then by using Proposition \ref{Sharpness of embedding, discrete form},
the two function spaces in the corresponding embedding relations are actually the same.
We have now completed the proof of \ $A\subset X.$

\subsection{The proof of $A\cap X_1=A_1$.}$~$\\
In order to show  $A\cap X_1=A_1$, we need the following proposition for reduction purpose.
\begin{proposition}[For reduction, discrete form]\label{For reduction, discrete form}
Suppose $0< q, q_1, q_2\leq \infty$, $s> 0$.
If
\begin{equation*}
\frac{1}{q}\leq\frac{1}{q_1},~\frac{1}{q}\leq\frac{1}{q_2},~~
1+\frac{1}{q}<\frac{1}{q_1}+\frac{1}{q_2}+\frac{s}{n},
\end{equation*}
then
\begin{equation*}
l(q_1, s)\ast l(q_2, s) \subset l(q, s).
\end{equation*}
\end{proposition}
\begin{proof}
By the assumption, we have
\begin{equation*}
\frac{1}{q_{1}}+\frac{1}{q_{2}}-\frac{1}{q}\geq 0.
\end{equation*}%
We denote
\begin{equation*}
\frac{1}{r}=\left( \frac{1}{q_{1}}+\frac{1}{q_{2}}-\frac{1}{q}\right) \wedge
1
\end{equation*}
and let $t=r^{\prime }$. Then $r\in \lbrack 1,\infty ]$ and
\begin{equation*}
\frac{1}{r}+\frac{1}{q}\leq \frac{1}{q_{1}}+\frac{1}{q_{2}}.
\end{equation*}
Using Proposition \ref{Integral capability of weight, discrete form}, we deduce that
\begin{equation*}
\left\Vert \frac{\langle k\rangle ^{s}}{\langle k-j\rangle ^{s}\langle
j\rangle ^{s}}\right\Vert _{l_{j}^{t}}\lesssim 1
\end{equation*}
uniformly for $k\in \mathbb{Z}^{n}$.

For $r\neq \infty$,
\begin{equation*}
\begin{split}
&\bigg(\sum_{k\in \mathbb{Z}^n}\Big(\sum_{j\in \mathbb{Z}^n}a_{k-j}b_j\Big)
^q \langle k\rangle^{sq}\bigg)^{{1}/{q}} \\
=& \bigg(\sum_{k\in \mathbb{Z}^n}\Big(\sum_{j\in \mathbb{Z}^n}a_{k-j}\langle
k-j\rangle^{s}b_j\langle j\rangle^{s} \frac{\langle k\rangle^{s}}{\langle
k-j\rangle^{s}\langle j\rangle^{s}}\Big)^q \bigg)^{{1}/{q}} \\
\lesssim & \sup_{k\in \mathbb{Z}^n}\left\|\frac{\langle k\rangle^{s}}{
\langle k-j\rangle^{s}\langle j\rangle^{s}}\right\|_{l_j^t} \bigg(\sum_{k\in
\mathbb{Z}^n}\Big(\sum_{j\in \mathbb{Z}^n}a^r_{k-j}\langle
k-j\rangle^{sr}b^r_j\langle j\rangle^{sr} \Big)^{{q}/{r}} \bigg)^{{1}/{q}} \\
\lesssim & \bigg(\sum_{k\in \mathbb{Z}^n}\Big(\sum_{j\in \mathbb{Z}
^n}a^r_{k-j}\langle k-j\rangle^{sr}b^r_j\langle j\rangle^{sr} \Big)^{{q}/{r}} \bigg)^{{1}/{q}} =\left\|\{a_k^r\langle k\rangle^{sr}\}\ast
\{b_k^r\langle k\rangle^{sr}\}\right\|_{l^{{q}/{r}}}^{{1}/{r}}.
\end{split}%
\end{equation*}
By the fact that $1/{q}\leq 1/{q_1}, 1/{q}\leq 1/{q_2}$ and ${1}/{r}+{1}/{q}\leq {1}/{q_1}+{1}/{q_2}$, we have,
by Proposition \ref{Sharpness of the Young's inequality, discrete form},
\begin{equation*}
\begin{split}
\left\|\{a_k^r\langle k\rangle^{sr}\}\ast \{b_k^r\langle
k\rangle^{sr}\}\right\|_{l^{{q}/{r}}}^{{1}/{r}} &\lesssim
\left\|\{a_k^r\langle k\rangle^{sr}\}\right\|_{l^{{q_1}/{r}}}^{{1}/{r}}\left\|\{b_k^r\langle k\rangle^{sr}\}\right\|_{l^{{q_2}/{r}}}^{{1}/{r}} \\
&= \left\|\{a_k\langle k\rangle^{s}\}\right\|_{l^{q_1}}\left\|\{b_k\langle
k\rangle^{s}\}\right\|_{l^{q_2}} \\
&= \|\{a_k\}\|_{l(q_1,s)}\|\{b_k\}\|_{l(q_2,s)}.
\end{split}%
\end{equation*}

For $r=\infty $, we have $q=q_{1}=q_{2}=\infty $ and
\begin{equation*}
\left\Vert \frac{\langle k\rangle ^{s}}{\langle k-j\rangle ^{s}\langle
j\rangle ^{s}}\right\Vert _{l_{j}^{1}}\lesssim 1
\end{equation*}
uniformly for $k\in \mathbb{Z}^{n}$. So for any $k\in \mathbb{Z}^{n}$, we
have
\begin{equation*}
\begin{split}
\sum_{j\in \mathbb{Z}^{n}}a_{k-j}b_{j}\langle k\rangle ^{s}=& \sum_{j\in
\mathbb{Z}^{n}}a_{k-j}\langle k-j\rangle ^{s}b_{j}\langle j\rangle ^{s}\frac{
\langle k\rangle ^{s}}{\langle k-j\rangle ^{s}\langle j\rangle ^{s}} \\
\lesssim & \Vert \{a_{k}\}\Vert _{l(\infty ,s)}\Vert \{b_{k}\}\Vert
_{l(\infty ,s)}\sum_{j\in \mathbb{Z}^{n}}\frac{\langle k\rangle ^{s}}{
\langle k-j\rangle ^{s}\langle j\rangle ^{s}} \\
\lesssim & \Vert \{a_{k}\}\Vert _{l(\infty ,s)}\Vert \{b_{k}\}\Vert
_{l(\infty ,s)}.
\end{split}
\end{equation*}
Hence
\begin{equation*}
\Vert \{a_{k}\}\ast \{b_{k}\}\Vert _{l(\infty ,s)}\lesssim \Vert
\{a_{k}\}\Vert _{l(\infty ,s)}\Vert \{b_{k}\}\Vert _{l(\infty ,s)},
\end{equation*}
which completes the proof of Proposition 3.1.
\end{proof}

Now we can return to prove $A\cap X_{1}= A_{1}.$
First, the inclusion  $A\cap X_{1}\subset A_{1}$ is obvious, since
this fact\textbf{\ }can be verified directly by
\begin{equation*}
A\cap X_{1}=A\cap A_{1}\subset A_{1}.
\end{equation*}%
It remains to show \ $A_{1}\subset A\cap X_{1}$. To this end, we
only need to show
\begin{equation*}
(\mathbf{q},\mathbf{s})\in A_{1}\Longrightarrow l(q_{1},s_{1})\ast
l(q_{2},s_{2})\subset l(q,s).
\end{equation*}

We will consider three different cases:\\
Case 1: $s\geq 0, s_1\geq 0, s_2\geq 0;$ \\
Case 2: $s<0, s_1> 0, s_2>0;$\\
Case 3: $s<0, s_1\leq 0, s_2\geq 0$ or $s<0, s_1\geq 0, s_2\leq 0$.

We point out that Case 2 and Case 3 can be reduced to Case 1.

In fact, in Case 2 one can choose
\begin{equation*}
\frac{1}{\bar{q}}=\max\left\{\frac{1}{q}+\frac{s}{n}, 0\right\},~\bar{s}>0,
\end{equation*}
so that, by Proposition \ref{Sharpness of embedding, discrete form},
\begin{equation*}
l(\bar{q},\bar{s})\subset l(q,s)
\end{equation*}
and the new index group $(\bar{\textbf{q}},\bar{\textbf{s}})=(\bar{q},q_1,q_2,\bar{s},s_1,s_2)$ belongs to Case 1.
Hence, the conclusion with Case 2 can be deduced by that with Case 1 and the embedding of $l(\bar{q},\bar{s})$ and $l(q,s)$.

In Case 3 we may assume $s<0, s_1\leq 0, s_2\geq 0$ for the symmetry of $\ s_1, s_2$.
We now easily verify that
\begin{equation*}
\Big(\frac{1}{q_1}+\frac{s_1}{n}\Big)\vee 0+\Big(\frac{1}{q_1'}+\frac{-s_1}{n}\Big)\vee 0\leq \Big(\frac{1}{q}+\frac{s}{n}\Big)\vee 0+\Big(\frac{1}{q'}+\frac{-s}{n}\Big)\vee 0.
\end{equation*}
Combining it with
\begin{equation*}
1+\Big(\frac{1}{q}+\frac{s}{n}\Big)\vee 0<\Big(\frac{1}{q_1}+\frac{s_1}{n}\Big)\vee 0+\Big(\frac{1}{q_2}+\frac{s_2}{n}\Big)\vee 0,
\end{equation*}
we have
\begin{equation*}
1+\Big(\frac{1}{q_1'}+\frac{-s_1}{n}\Big)\vee 0<\Big(\frac{1}{q'}+\frac{-s}{n}\Big)\vee 0+\Big(\frac{1}{q_2}+\frac{s_2}{n}\Big)\vee 0.
\end{equation*}
The new index
$(\bar{\bar{\textbf{q}}},\bar{\bar{\textbf{s}}})=(q'_1,q',q_2,-s_1,-s,s_2)$ then belongs to Case 1.
If one can deduce
\begin{equation*}
l(q', -s)\ast l(q_2, s_2) \subset l(q_1', -s_1)
\end{equation*}
from Case 1,
the conclusion associated with Case 3 follows easily by a dual argument.

Now, we only need to handle Case 1.
By the spirit of Proposition \ref{Sharpness of embedding, discrete form}, one can choose $\ \widetilde{q_1},~\widetilde{q_2}\in (0,\infty],$ such that
\begin{equation*}
\frac{1}{\widetilde{q_1}}+\frac{s}{n}\leqslant \frac{1}{q_1}+\frac{s_1}{n}
\end{equation*}
with strict inequality if $s<s_1$,
\begin{equation*}
\frac{1}{\widetilde{q_2}}+\frac{s}{n}\leqslant \frac{1}{q_2}+\frac{s_2}{n}
\end{equation*}
with strict inequality if $s<s_2$,
and
\begin{equation*}
\begin{cases}
s\geq 0,\\
\frac{1}{q}+\frac{s}{n}\leq \frac{1}{\widetilde{q_1}}+\frac{s}{n},~\frac{1}{q}+\frac{s}{n}\leq \frac{1}{\widetilde{q_2}}+\frac{s}{n},\\
1+\frac{1}{q}+\frac{s}{n}<\frac{1}{\widetilde{q_1}}+\frac{s}{n}+\frac{1}{\widetilde{q_2}}+\frac{s}{n}.
\end{cases}
\end{equation*}
Then, by Proposition \ref{Sharpness of embedding, discrete form}, we obtain
\begin{equation*}
l(q_1,s_1)\subset l(\widetilde{q_1},s),~l(q_2,s_2)\subset l(\widetilde{q_2},s).
\end{equation*}
Moreover, we can use Proposition \ref{Sharpness of the Young's inequality, discrete form} or Proposition \ref{For reduction, discrete form} to deduce
\begin{equation*}
l(\widetilde{q_1}, s)\ast l(\widetilde{q_2}, s) \subset l(q, s).
\end{equation*}
The desired conclusion now follows by an embedding argument.

\subsection{The proof of $A\cap X_2=A_2$.}$~$\\
We want to show the inclusion $A\cap X_2\subset A_2$.

Firstly, we show
\begin{equation*}
(\textbf{q},\textbf{s})\in A\cap X_2\cap \left\{(\textbf{q},\textbf{s})\in X_2:  \frac{1}{q}+\frac{s}{n}= \frac{1}{q_1}+\frac{s_1}{n} \right\}
\Longrightarrow
\begin{cases}
s=s_1=s_2=0,\\
q=q_1,~q_2=1.
\end{cases}
\end{equation*}
In this case, we have $(q,s)=(q_1,s_1)$
and
\begin{equation*}
\frac{1}{q_2}+\frac{s_2}{n}=1.
\end{equation*}
Using (\ref{Theorem 1, conclusion}) and H\"{o}lder's inequality, we obtain
\begin{equation}\label{for proof, 1}
  \begin{split}
    \sum_{k\in \mathbb{Z}^n}\sum_{j\in \mathbb{Z}^n}a_jb_{k-j}c_k
    \lesssim &
    \|\{a_k\}\ast\{b_k\}\|_{l(q,s)}\|\{c_k\}\|_{l(q',-s)}
    \\
    \lesssim &
    \|\{a_k\}\|_{l(q_1,s_1)}\|\{c_k\}\|_{l(q',-s)}\|\{b_k\}\|_{l(q_2,s_2)}.
  \end{split}
\end{equation}
For $\frac{1}{q_1}+\frac{s_1}{n}> 0,$
we let
\begin{equation*}
a_{k,N}=
\begin{cases}
\frac{1}{(2N)^{s_1+{n}/{q_1}}}, &|k|\leq 2N
\\
0, &\text{otherwise, }
\end{cases}
\hspace{10mm}
c_{k,N}=
\begin{cases}
\frac{\langle k\rangle^t}{N^{t-s+{n}/{q'}}},&|k|\leq 2N
\\
0,&\text{otherwise }
\end{cases}
\end{equation*}
for some $t$ satisfying
\begin{equation*}
t>-n,~\frac{1}{q'}+\frac{t-s}{n}> 0.
\end{equation*}
It is easy to see that
\begin{equation*}
\|\{a_{k,N}\}\|_{l(q_1,s_1)}\sim \|c_{k,N}\|_{l(q',-s)}\sim 1.
\end{equation*}
We then use (\ref{for proof, 1}) to deduce
\begin{equation*}
\begin{split}
\|\{b_k\}\|_{l(q_2,s_2)}
\gtrsim &
\sum_{k\in \mathbb{Z}^n}\sum_{j\in \mathbb{Z}^n}a_{j,N}b_{k-j}c_{k,N}
\\
\gtrsim &
\frac{1}{(2N)^{s_1+\frac{n}{q_1}}} \sum_{|k|\leq N}c_{k,N}\sum_{|j|\leq 2N}b_{k-j}
\\
\gtrsim &
\frac{1}{(2N)^{s_1+\frac{n}{q_1}}} \sum_{|k|\leq N}c_{k,N}\sum_{|j|\leq N}b_{j}
\\
\sim &
\frac{1}{(2N)^{s_1+\frac{n}{q_1}}}\cdot\frac{N^{t+n}}{N^{t-s+\frac{n}{q'}}}\sum_{|j|\leq N}b_{j}
\sim
\sum_{|j|\leq N}b_{j}.
\end{split}
\end{equation*}
Letting $N\rightarrow \infty$, we obtain the embedding relationship
\begin{equation*}
l(q_2,s_2)\subset l(1,0),
\end{equation*}
and use Proposition \ref{Sharpness of embedding, discrete form} to deduce
\begin{equation*}
s_2=0,~q_2=1.
\end{equation*}
Recalling $0\leq s_1+s_2,~ s=s_1,~ s\leq s_2$,
we deduce
\begin{equation*}
s=s_1=s_2=0.
\end{equation*}
For $\frac{1}{q_1}+\frac{s_1}{n}\leq 0,$ we have $\frac{1}{q'}+\frac{-s}{n}\geq 1>0.$
We can rewrite
\begin{equation*}
\sum_{j\in \mathbb{Z}^n}a_j\sum_{k\in \mathbb{Z}^n}c_kb_{k-j}=\sum_{k\in \mathbb{Z}^n}c_k\sum_{j\in \mathbb{Z}^n}a_jb_{k-j}
\end{equation*}
and use the same argument as above to deduce the conclusion.

Next, we can show that
\begin{equation*}
(\textbf{q},\textbf{s})\in A\cap X_2\cap \left\{(\textbf{q},\textbf{s})\in X_2: ~ \frac{1}{q}+\frac{s}{n}= \frac{1}{q_2}+\frac{s_2}{n} \right\}
\Longrightarrow
\begin{cases}
s=s_1=s_2=0,\\
q=q_2,~q_1=1,
\end{cases}
\end{equation*}
by using the same argument above.

Finally, we show
\begin{equation*}
(\textbf{q},\textbf{s})\in A\cap X_2\cap \left\{(\textbf{q},\textbf{s})\in X_2: ~ \frac{1}{q}+\frac{s}{n}=0 \right\}
\Longrightarrow
\begin{cases}
s=s_1=s_2=0,\\
~q=\infty,~\frac{1}{q_1}+\frac{1}{q_2}=1.
\end{cases}
\end{equation*}
In this case, we have
\begin{equation*}
1=\frac{1}{q_1}+\frac{s_1}{n}+\frac{1}{q_2}+\frac{s_2}{n}\Longleftrightarrow \frac{1}{q_1'}+\frac{-s_1}{n}=\frac{1}{q_2}+\frac{s_2}{n}.
\end{equation*}
Also, we have
\begin{equation*}
\frac{1}{q'}+\frac{-s}{n}=1.
\end{equation*}
So, a dual argument gives that
\begin{equation*}
l(q', -s)\ast l(q_2, s_2) \subset l(q_1', -s_1).
\end{equation*}
With the same argument as we used above, we obtain the facts
\begin{equation*}
-s_1=-s=-s_2=0,~ q'=1,~ q_2=q_1'.
\end{equation*}
We have now completed the proof for $A_{2}\supset A\cap X_{2}$.

To prove the opposite inclusion $A_{2}\subset A\cap X_{2}$, we need
to show
\begin{equation*}
(\textbf{q},\textbf{s})\in A_2\Longrightarrow l(q_1, s_1)\ast l(q_2, s_2) \subset l(q, s).
\end{equation*}
This fact can be directly verified by Proposition \ref{Sharpness of the Young's inequality, discrete form}.

\subsection{The proof of $A\cap X_3=A_3$.}$~$\\

We want to prove the inclusion\textbf{\ }$A\cap X_{3}\subset A_{3}$. We only need to show
\begin{equation*}
(\textbf{q},\textbf{s})\in A\cap X_3
\Longrightarrow
\begin{cases}
\frac{1}{q_1}+\frac{1}{q_2}=1,~s_1+s_2=0,\\
0\leq \frac{1}{q_1}+\frac{s_1}{n},\frac{1}{q_2}+\frac{s_2}{n}.
\end{cases}
\end{equation*}
In this case, we have
\begin{equation*}
1=\Big(\frac{1}{q_1}+\frac{s_1}{n}\Big)\vee 0+\Big(\frac{1}{q_2}+\frac{s_2}{n}\Big)\vee 0.
\end{equation*}
Hence, the inequalities
\begin{equation*}
1\leq \frac{1}{q_1}+\frac{s_1}{n}+\frac{1}{q_2}+\frac{s_2}{n}\leq \Big(\frac{1}{q_1}+\frac{s_1}{n}\Big)\vee 0+\Big(\frac{1}{q_2}+\frac{s_2}{n}\Big)\vee 0=1
\end{equation*}
implies
\begin{equation*}
1=\frac{1}{q_1}+\frac{s_1}{n}+\frac{1}{q_2}+\frac{s_2}{n}=
\Big(\frac{1}{q_1}+\frac{s_1}{n}\Big)\vee 0+\Big(\frac{1}{q_2}+\frac{s_2}{n}\Big)\vee 0.
\end{equation*}
So the desired conclusion holds.

To show the opposite inclusion\textbf{\ }$A_{3}\subset A\cap X_{3}$,
it suffices to show
\begin{equation}\label{for proof, 8}
(\textbf{q},\textbf{s})\in A_3\Longrightarrow l(q_1, s_1)\ast l(q_2, s_2) \subset l(q, s).
\end{equation}
In fact, we have
\begin{equation*}
1+\Big(\frac{1}{q_1'}+\frac{-s_1}{n}\Big)\vee 0<\Big(\frac{1}{q'}+\frac{-s}{n}\Big)\vee 0+\Big(\frac{1}{q_2}+\frac{s_2}{n}\Big)\vee 0
\end{equation*}
in this case.
Hence the proof of (\ref{for proof, 8}) can be completed by a dual argument and the same
method we used in the proof of $A\cap X_{1}=A_{1}$.
\subsection{The proof of $A\cap X_4=A_4$.}$~$\\
We want to prove the inclusion\textbf{\ }$A\cap X_{4}\subset A_{4}$.

Firstly, we show
\begin{equation*}
(\textbf{q},\textbf{s})\in A\cap X_4
\Longrightarrow
\frac{1}{q}\leq \frac{1}{q_1}+\frac{1}{q_2}.
\end{equation*}
The proof for case $q=\infty$ is obvious, so we only treat the case $q<\infty$.

Fix $N\in \mathbb{N}$.
Let
\begin{equation*}
a_{k,N}=
\begin{cases}
\langle k \rangle^{-s_1-\frac{n}{q_1}},&|k|\leq N,\\
0,&\text{otherwise},
\end{cases}
\hspace{10mm}
b_{k,N}=
\begin{cases}
\langle k \rangle^{-s_2-\frac{n}{q_2}},&|k|\leq N,\\
0,&\text{otherwise}.
\end{cases}
\end{equation*}
It is easy to see that
\begin{equation*}
\|\{a_{k,N}\}\|_{l(q_1,s_1)}\sim(\ln N)^{\frac{1}{q_1}},~\|\{b_{k,N}\}\|_{l(q_2,s_2)}\sim(\ln N)^{\frac{1}{q_2}}.
\end{equation*}
On the other hand,
\begin{equation*}
\begin{split}
\|\{a_{k,N}\}\ast \{b_{k,N}\}\|^q_{l(q, s)}
= &
\sum_{k\in \mathbb{Z}^n}\Big(\sum_{j\in \mathbb{Z}^n}a_{k-j,N}b_{j,N}\Big)^{q}\langle k\rangle^{sq}
\\
= &
\sum_{|k|\leq N}\Big(\sum_{|j|\leq N}a_{k-j,N}b_{j,N}\Big)^{q}\langle k\rangle^{sq}
\\
\gtrsim  &
\sum_{|k|\leq N}\Big(\sum_{|j-\frac{1}{2}k|\leq |\frac{1}{4}k|}\langle k-j\rangle^{-s_1-\frac{n}{q_1}} \langle j\rangle^{-s_2-\frac{n}{q_2}}\Big)^{q}\langle k\rangle^{sq}
\\
\gtrsim  &
\sum_{|k|\leq N}\Big(\sum_{|j-\frac{1}{2}k|\leq |\frac{1}{4}k|}\langle k\rangle^{-s_1-\frac{n}{q_1}-s_2-\frac{n}{q_2}}\Big)^{q}\langle k\rangle^{sq}
\\
\gtrsim  &
\sum_{|k|\leq N}\Big(\langle k\rangle^n\cdot\langle k\rangle^{-s_1-\frac{n}{q_1}-s_2-\frac{n}{q_2}}\Big)^{q}\langle k\rangle^{sq}.
\end{split}
\end{equation*}
Recalling
\begin{equation*}
1+\frac{1}{q}+\frac{s}{n}=\frac{1}{q_1}+\frac{s_1}{n}+\frac{1}{q_2}+\frac{s_2}{n},
\end{equation*}
we deduce that
\begin{equation*}
\|\{a_{k,N}\}\ast \{b_{k,N}\}\|^q_{l(q, s)}
\gtrsim
\sum_{|k|\leq N}\left(\langle k\rangle^{-s-{n}/{q}}\right)^{q}\langle k\rangle^{sq}
=
\sum_{|k|\leq N}\langle k\rangle^{-n}\sim \ln N,
\end{equation*}
which gives
\begin{equation*}
\|\{a_{k,N}\}\ast \{b_{k,N}\}\|_{l(q, s)}
\gtrsim (\ln N)^{{1}/{q}}.
\end{equation*}
Hence, the obtained inequality
\begin{equation*}
(\ln N)^{{1}/{q}}\lesssim (\ln N)^{{1}/{q_1}}\cdot (\ln N)^{{1}/{q_2}}
\end{equation*}
implies that
\begin{equation*}
\frac{1}{q}\leq \frac{1}{q_1}+\frac{1}{q_2}.
\end{equation*}

Now, we show
\begin{equation*}
(\textbf{q},\textbf{s})\in A\cap X_4\cap \left\{(\textbf{q},\textbf{s})\in X_4: ~s=s_1~or~s=s_2 \right\}
\Longrightarrow q\neq \infty,~q_1,q_2\neq 1.
\end{equation*}
We will prove this fact by a contradiction argument. If $(\textbf{q},\textbf{s})$ satisfies
\begin{equation}\label{for proof, 2}
\begin{cases}
0<s\leq s_1,~s=s_2,\\
1=\frac{1}{q_1}+\frac{1}{q_2}+\frac{s_1}{n},q=\infty,\\
\frac{s}{n}<\frac{1}{q_1}+\frac{s_1}{n},~q_2<\infty.\\
\end{cases}
\end{equation}
We define
\begin{equation*}
a_k=\langle k\rangle^{-{n}/{q_1}-s_1}\left(1+\ln\langle k\rangle\right)^{\epsilon_1}
\end{equation*}
for all $k\in \mathbb{Z}^n$, where
\begin{equation*}
\epsilon_1=
\begin{cases}
\text{a real number such that}~\ q_1\epsilon_1<-1,&\text{if}~q_1<\infty,
\\
0,&\text{if}~q_1=\infty.
\end{cases}
\end{equation*}
For any $N\in \mathbb{Z}^n$, we define
\begin{equation*}
b_{k,N}=
\frac{\langle N-k\rangle^{-\frac{n}{q_2}}}{\langle k\rangle^{s_2}}\left(1+\ln\langle N-k\rangle\right)^{\epsilon_2}
\end{equation*}
for all $k\in \mathbb{Z}^n$, where
$\ \epsilon_2$
is a real number to be chosen later such that $q_2\epsilon_2<-1$.
By a direct calculation, we have
$\{a_k\}_{k\in \mathbb{Z}^n}\in l(q_1,s_1)$, and
\begin{equation*}
\|\{b_{k,N}\}\|_{l(q_2,s_2)}\lesssim 1
\end{equation*}
uniformly for all $N\in \mathbb{Z}^n$.
But we find
\begin{equation*}
\begin{split}
&\sum_{|j-N|\leq {|N|}/{2}}a_{N-j}b_{j,N}\langle N\rangle^s
\\
=&
\sum_{|j-N|\leq {|N|}/{2}}\langle N-j\rangle^{-{n}/{q_1}-s_1}\left(1+\ln\langle N-j\rangle\right)^{\epsilon_1+\epsilon_2}
\cdot
\frac{\langle N-j\rangle^{-{n}/{q_2}}}{\langle j\rangle^{s_2}}%\left(1+\ln\langle N-j\rangle\right)^{\epsilon_2}
\langle N\rangle^s
\\
\sim &
\sum_{|j-N|\leq {|N|}/{2}}\langle N-j\rangle^{-{n}/{q_1}-s_1}\left(1+\ln\langle N-j\rangle\right)^{\epsilon_1+\epsilon_2}
\cdot
\langle N-j\rangle^{-{n}/{q_2}}%\left(1+\ln\langle N-j\rangle\right)^{\epsilon_2}
\\
=&
\sum_{|j|\leq {|N|}/{2}}\langle j\rangle^{-{n}/{q_1}-s_1}\left(1+\ln\langle j\rangle\right)^{\epsilon_1+\epsilon_2}
\cdot
\langle j\rangle^{-{n}/{q_2}}%\left(1+\ln\langle j\rangle\right)^{\epsilon_2}
\\
=&
\sum_{|j|\leq {|N|}/{2}}\langle j\rangle^{-n}\left(1+\ln\langle j\rangle\right)^{\epsilon_1+\epsilon_2}.
\end{split}
\end{equation*}
Since (\ref{for proof, 2}) implies
\begin{equation*}
\frac{1}{q_1}+\frac{1}{q_2}<1,
\end{equation*}
we can choose $\ \epsilon_1$ and $\ \epsilon_2$ such that
\begin{equation*}
\epsilon_1+\epsilon_2>-1.
\end{equation*}
On the other hand, it is easy to check that
\begin{equation*}
\|\{a_k\}\|_{l(q_1,s_1)}\cdot \|\{b_{k,N}\}\|_{l(q_2,s_2)}\lesssim 1
\end{equation*}
uniformly on $N$. This leads to a contradiction
\begin{equation*}
\begin{split}
1\gtrsim \|\{a_k\}\ast\{b_{k,N}\}\|_{l(\infty,s)}
\geq &
(\{a_k\}\ast\{b_{k,N}\})(N)\langle N\rangle^s
\\
=&
\sum_{j\in \mathbb{Z}^n}a_{N-j}b_{j,N}\langle N\rangle^s
\\
\gtrsim &
\sum_{|j-N|\leq \frac{|N|}{2}}a_{N-j}b_{j,N}\langle N\rangle^s
\\
\gtrsim &
\sum_{|j|\leq \frac{|N|}{2}}\langle j\rangle^{-n}\left(1+\ln\langle j\rangle\right)^{\epsilon_1+\epsilon_2}\rightarrow \infty
\end{split}
\end{equation*}
as $|N|\rightarrow \infty.$

If $\ s=s_1$, $q=\infty$, one can also deduce a contradiction by the same argument as above.
Also, the case $q_1=1$ or $q_2=1$ can be handled by a dual argument.
\\

To complete the proof of \ $A_{4}=A\cap X_{4}$, it now remains to
show $A_{4}\subset A\cap X_{4}$.

Firstly, we show
\begin{equation*}
(\textbf{q},\textbf{s})\in A_4\cap \{(\textbf{q},\textbf{s})\in A_4:q\neq \infty, q_1,q_2\neq 1\}\Longrightarrow l(q_1, s_1)\ast l(q_2, s_2) \subset l(q, s).
\end{equation*}
This conclusion will be obtained by using the fact that the discrete form of
the Young-type inequalities can be deduced from the continuous form of
Young-type inequalities. To this end, we establish the following
proposition which will also play a pivotal action in the proof of Theorem
\ref{Sharpness of weighted Young's inequality, power weight}.
\begin{proposition}[Implication method]\label{Implication method}
Suppose $1\leq q, q_1, q_2\leq \infty$, $s, s_1, s_2\in \mathbb{R}$. Then
\begin{equation*}
\mathbb{L}(q_1, s_1)\ast \mathbb{L}(q_2, s_2) \subset \mathbb{L}(q, s)
\end{equation*}
implies
\begin{equation*}
l(q_1, s_1)\ast l(q_2, s_2) \subset l(q, s).
\end{equation*}
\end{proposition}
\begin{proof}
We denote the unit cube
\begin{equation*}
Q=\left\{x\in \mathbb{R}^n:~-1/2\leq x_j \leq 1/2, j=1,2,\cdots,n \right\}.
\end{equation*}

Given two positive sequences $\{a_k\}_{k\in \mathbb{Z}^n}$ and $\{b_k\}_{k\in \mathbb{Z}^n}$ defined on $\mathbb{Z}^n$,
we define the companion functions
\begin{equation*}
f(x)=\sum_{k\in \mathbb{Z}^n}a_k\chi_{Q\backslash \frac{1}{8}Q}(x-k)
\end{equation*}
and
\begin{equation*}
g(x)=\sum_{k\in \mathbb{Z}^n}b_k\chi_{Q\backslash \frac{1}{8}Q}(x-k).
\end{equation*}
It is easy to check that
\begin{equation*}
\|f\|_{\mathbb{L}(q_1, s_1)}
\sim
\left(\sum_{k\in \mathbb{Z}^n}a_k^{q_1}\langle k\rangle^{s_1q_1} \right)^{{1}/{q_1}}
=\|\{a_k\}\|_{l_k(q_1,s_1)}
\end{equation*}
and
\begin{equation*}
\|g\|_{\mathbb{L}(q_2, s_2)}
\sim
\left(\sum_{k\in \mathbb{Z}^n}b_k^{q_2}\langle k\rangle^{s_2q_2} \right)^{{1}/{q_2}}
=\|\{b_k\}\|_{l_k(q_2,s_2)}.
\end{equation*}
On the other hand,
\begin{eqnarray*}
&&\Vert f\ast g\Vert _{\mathbb{L}(q,s)}^{q}=\int_{\mathbb{R}^{n}}\Big|\int_{%
\mathbb{R}^{n}}f(x-y)g(y)dy\Big|^{q}|x|^{sq}dx \\
&=&\sum_{k\in \mathbb{Z}^{n}}\int_{Q+k}\Big|\sum_{j\in \mathbb{Z}%
^{n}}\int_{Q+j}f(x-y)g(y)dy\Big|^{q}|x|^{sq}dx \\
&=&\sum_{k\in \mathbb{Z}^{n}}\int_{Q}\Big|\sum_{j\in \mathbb{Z}%
^{n}}\int_{Q}f(k-j+x-y)g(y+j)dy\Big|^{q}|x+k|^{sq}dx \\
&\gtrsim &\sum_{k\in \mathbb{Z}^{n}}\int_{\frac{1}{2}Q}\Big|\sum_{j\in
\mathbb{Z}^{n}}\int_{\frac{1}{2}Q}f(k-j+x-y)g(y+j)dy\Big|^{q}|x+k|^{sq}dx \\
&=&\sum_{k\in \mathbb{Z}^{n}}\int_{\frac{1}{2}Q}\Big|\sum_{j\in \mathbb{Z}%
^{n}}b_{j}\int_{\frac{1}{2}Q\backslash \frac{1}{8}Q}f(k-j+x-y)dy\Big|%
^{q}|x+k|^{sq}dx \\
&\geq &\sum_{k\in \mathbb{Z}^{n}}\int_{\frac{1}{2}Q}\Big|\sum_{j\in \mathbb{Z%
}^{n}}a_{k-j}b_{j}\int_{(\frac{1}{2}Q\backslash \frac{1}{8}Q)\cap
(x+Q\backslash \frac{1}{8}Q)}dy\Big|^{q}|x+k|^{sq}dx \\
&=&\sum_{k\in \mathbb{Z}^{n}}\int_{\frac{1}{2}Q}\Big|\sum_{j\in \mathbb{Z}%
^{n}}a_{k-j}b_{j}|(\frac{1}{2}Q\backslash \frac{1}{8}Q)\cap (x+Q\backslash
\frac{1}{8}Q)|\Big|^{q}|x+k|^{sq}dx \\
&\gtrsim &\sum_{k\in \mathbb{Z}^{n}}\int_{\frac{1}{2}Q}\Big|\sum_{j\in
\mathbb{Z}^{n}}a_{k-j}b_{j}\Big|^{q}|x+k|^{sq}dx \\
&\gtrsim &\sum_{k\in \mathbb{Z}^{n}}\Big|\sum_{j\in \mathbb{Z}%
^{n}}a_{k-j}b_{j}\Big|^{q}\langle k\rangle ^{sq}=\Vert \{a_{k}\}\ast
\{b_{k}\}\Vert _{l(q,s)}^{q}.
\end{eqnarray*}
\end{proof}
Now the desired conclusion $l(q_1, s_1)\ast l(q_2, s_2) \subset l(q, s)$ follows directly from the above proposition and Theorem A.

Next, we show
\begin{equation*}
(\textbf{q},\textbf{s})\in A_4\cap \{(\textbf{q},\textbf{s})\in A_4:~q=\infty,~s<s_1,~s<s_2\}
\end{equation*}
implies
\begin{equation*}
l(q_1, s_1)\ast l(q_2, s_2) \subset l(\infty, s).
\end{equation*}
In this case, we have
\begin{equation*}
s>0,~\frac{1}{q_1}+\frac{1}{q_2}<1.
\end{equation*}
We denote
\begin{equation*}
\frac{1}{t}=1-\frac{1}{q_1}-\frac{1}{q_2}=\frac{s_1+s_2-s}{n}.
\end{equation*}
For any $k\in \mathbb{Z}^n$, we use Proposition \ref{Integral capability of weight, discrete form} to deduce
\begin{equation*}
\begin{split}
\sum_{j\in \mathbb{Z}^n}a_{k-j}b_j\langle k\rangle^s
= &
\sum_{j\in \mathbb{Z}^n}a_{k-j}\langle k-j\rangle^{s_1}b_j\langle j\rangle^{s_2}\frac{\langle k\rangle^s}{\langle k-j\rangle^{s_1}\langle j\rangle^{s_2}}
\\
\lesssim  &
\|\{a_k\langle k\rangle^{s_1}\}\|_{l^{q_1}}
\cdot
\|\{b_k\langle k\rangle^{s_2}\}\|_{l^{q_2}}
\cdot
\left\|\left\{\frac{\langle k\rangle^s}{\langle k-j\rangle^{s_1}\langle j\rangle^{s_2}}\right\} \right\|_{l_j^t}
\\
\lesssim  &
\|\{a_k\}\|_{l(q_1,s_1)}
\cdot
\|\{b_k\}\|_{l(q_2,s_2)}.
\end{split}
\end{equation*}
Finally, by a dual argument, we can deduce the conclusion in the case
$q_{1}=1$ or $q_{2}=1$, $s<s_{1},s<s_{2}.$
The proof of Theorem \ref{Sharpness of weighted Young's inequality, discrete form} is completed.

\section{Continuous weighted Young's inequality---Proof of Theorem \ref{Sharpness of weighted Young's inequality, power weight}  and Theorem \ref{Sharpness of weighted Young's inequality, continuous form}. }

\subsection{The proof of Theorem \ref{Sharpness of weighted Young's inequality, power weight}.}$~$\\
We start our proof by showing the necessity \textbf{\ }
\begin{equation*}
\mathbf{\mathbb{L}(q_{1},s_{1})\ast \mathbb{L}(q_{2},s_{2})\subset \mathbb{L}
(q,s)\Longrightarrow (\mathbf{q},\mathbf{s})\in }A_{2}\cup A_{4}\mathbf{.}
\end{equation*}
Using a dilation argument, we first deduce
\begin{equation}\label{for proof, 3}
1+\frac{1}{q}+\frac{s}{n}=\frac{1}{q_1}+\frac{s_1}{n}+\frac{1}{q_2}+\frac{s_2}{n}.
\end{equation}
Then, we choose
\begin{equation*}
f(x)=\chi_{B(-2,1)},~g(x)=\chi_{B(2,1)}.
\end{equation*}
Observing
\begin{equation*}
(f\ast g)(x)\gtrsim 1
\end{equation*}
for $x\in B(0,{1}/{2}),$
we obtain
\begin{equation*}
\begin{split}
\int_{B(0,{1}/{2})}|x|^{qs}dx
\lesssim &
\int_{B(0,{1}/{2})}(f\ast g)^q(x)|x|^{qs}dx
\\
\lesssim &
\|f\ast g\|^q_{\mathbb{L}(q, s)}\lesssim \|f\|^q_{\mathbb{L}(q_1, s_1)}\cdot\|g\|^q_{\mathbb{L}(q_2, s_2)}\lesssim 1
\end{split}
\end{equation*}
for $q<\infty.$
It clearly yields
\begin{equation}\label{for proof, 4}
\begin{cases}
\frac{1}{q}+\frac{s}{n}>0, & \text{if~}q<\infty,
\\
s\geq 0, & \text{if~}q=\infty.
\end{cases}
\end{equation}
On the other hand, by Proposition \ref{Implication method} and Theorem \ref{Sharpness of weighted Young's inequality, discrete form}, we know
\begin{equation}\label{for proof, 5}
(\mathbf{q},\mathbf{s})\in \bigcup_{i=1}^4A_i.
\end{equation}
Combining with (\ref{for proof, 3}), (\ref{for proof, 4}) and (\ref{for proof, 5}), we conclude
\begin{equation*}
(\mathbf{q},\mathbf{s})\in A_2\cup A_4.
\end{equation*}

To prove the sufficiency of Theorem \ref{Sharpness of weighted Young's inequality, power weight}\textbf{\ }
\begin{equation*}
\mathbf{(\mathbf{q},\mathbf{s})\in }A_{2}\cup A_{4}\mathbf{\Longrightarrow
\mathbb{L}(q_{1},s_{1})\ast \mathbb{L}(q_{2},s_{2})\subset \mathbb{L}(q,s),}
\end{equation*}
we only need to handle the case $q=\infty$ or $q_1=1$ or $q_2=1$ in $\mathcal {A}_4.$
The other cases can be deduced directly by the classical Young's inequality and Kerman's result (See Theorem A).

By a dual method, we only need to show the proof for $q=\infty.$ In this case, we have
\begin{equation*}
s<s_1,~s<s_2
\end{equation*}
and
\begin{equation*}
1=\frac{1}{q_1}+\frac{1}{q_2}+\frac{s_1+s_2-s}{n}.
\end{equation*}
We now use Proposition \ref{Integral capability of weight, power weight} to deduce
\begin{equation*}
\begin{split}
&\int_{\mathbb{R}^n}f(x-y)g(y)|x|^sdy
\\
= &
\int_{\mathbb{R}^n}f(x-y)|x-y|^{s_1}g(y)|y|^{s_2}\frac{|x|^s}{|x-y|^{s_1}|y|^{s_2}}dy
\\
\lesssim &
\|f(x)|x|^{s_1}\|_{L_x^{q_1}}\cdot \|g(x)|x|^{s_2}\|_{L_x^{q_2}}\cdot \left\|\frac{|x|^s}{|x-y|^{s_1}|y|^{s_2}} \right\|_{L_y^{{n}/{(s_1+s_2-s)}}}
\\
\lesssim &
\|f\|_{\mathbb{L}(q_1, s_1)}\cdot \|g\|_{\mathbb{L}(q_2, s_2)}.
\end{split}
\end{equation*}
\subsection{The proof of Theorem \ref{Sharpness of weighted Young's inequality, continuous form}.}$~$\\
Firstly, we introduce some notations. Denote, for $i=1,2,3,4,$
\begin{equation*}
\begin{split}
B&=\left\{(\textbf{q},\textbf{s})\in [1,\infty]^3\times \mathbb{R}^3: ~ L(q_1, s_1)\ast L(q_2, s_2) \subset L(q, s)  \right\},\\
B_i&=A_i\cap \left\{(\textbf{q},\textbf{s})\in [1,\infty]^3\times \mathbb{R}^3: ~1+\frac{1}{q}\geq \frac{1}{q_1}+\frac{1}{q_2}  \right\}.
\end{split}
\end{equation*}
Using the same strategy for the proof of Theorem \ref{Sharpness of weighted Young's inequality, discrete form}, we only need to show
\begin{equation*}
B=\bigcup_{i=1}^4B_i.
\end{equation*}
The inclusion $B\subset \bigcup_{i=1}^{4}B_{i}$ is based on the following
two propositions.
\begin{proposition}[Relationship between discrete and continuous weighted Young's inequality]\label{Relationship between discrete weighted Young's inequality and continuous weighted Young's inequality}
Let $1\leq q, q_i \leq \infty, s_i\in \mathbb{R}$, for $i=1,2$.
Then the inclusion
\begin{equation*}
L(q_1, s_1)\ast L(q_2, s_2) \subset L(q, s)
\end{equation*}
implies
\begin{equation*}
l(q_1, s_1)\ast l(q_2, s_2) \subset l(q, s).
\end{equation*}
\end{proposition}
\begin{proof}
One can verify this conclusion by the implication method, which we have used in Proposition \ref{Implication method}.
However, we here give another proof based on the pointview of the modulation spaces.
In the definition of $V_{\phi }f$, since the choice of the window function is flexible, we
choose two window functions $\phi _{1}$ and $\phi _{2}$  so that
$\phi =\phi _{1}\phi _{2}$ is also a window function. An easy
computation gives that
\begin{equation*}
V_{\phi}(fg)(x,\xi)=\big(V_{\phi_1}f(x,\cdot)\ast V_{\phi_2}g(x,\cdot)\big)(\xi).
\end{equation*}
By the Minkowski inequality and H\"{o}lder's inequality, we obtain
\begin{equation*}
\begin{split}
\|fg\|_{M_{1,q}^s}=&\|V_{\phi}(fg)(x,\xi)\|_{L_{\xi}(q,s)(L_{x}^1)}
\\
=&\|\big(V_{\phi_1}f(x,\cdot)\ast V_{\phi_2}g(x,\cdot)\big)(\xi)\|_{L_{\xi}(q,s)(L_{x}^1)}
\\
\leq &
\Big\|\int_{\mathbb{R}^n}\|V_{\phi_1}f(x,\xi-\eta) V_{\phi_2}g(x,\eta)\|_{L_x^1} d\eta\Big\|_{L_{\xi}(q,s)}
\\
\leq &
\Big\|\int_{\mathbb{R}^n}\|V_{\phi_1}f(x,\xi-\eta)\|_{L_x^{2}} \|V_{\phi_2}g(x,\eta)\|_{L_x^2} d\eta\Big\|_{L_{\xi}(q,s)}.
\end{split}
\end{equation*}
We now use the continuous weighted Young's inequality to deduce
\begin{equation*}
\begin{split}
&\Big\|\int_{\mathbb{R}^n}\|V_{\phi_1}f(x,\xi-\eta)\|_{L_x^{2}} \|V_{\phi_2}g(x,\eta)\|_{L_x^2} d\eta\Big\|_{L_{\xi}(q,s)}
\\
&\qquad\lesssim
\left\|\|V_{\phi_1}f(x,\xi)\|_{L_x^{2}}\right\|_{L_{\xi}(q_1,s_1)}\cdot\left\|\|V_{\phi_2}g(x,\xi)\|_{L_x^{2}}\right\|_{L_{\xi}(q_2,s_2)}
\\
&\qquad=
\|f\|_{M_{2,q_1}^{s_1}}\cdot \|g\|_{M_{2,q_2}^{s_2}}.
\end{split}
\end{equation*}
In the next section, we will show
Proposition \ref{Relationship between the product on modulation spaces and the discrete weighted Young's inequality},
which says that the boundedness
$\left\Vert fg\right\Vert _{M_{1,q}^{s}} \lesssim \left\Vert f\right\Vert _{M_{1,q_{1}}^{s_{1}}}\left\Vert
g\right\Vert _{M_{1,q_{2}}^{s_{2}}}$
implies
$l(q_{1},s_{1})\ast l(q_{2},s_{2})\subset l(q,s).$
Since the proof for Proposition
\ref{Relationship between the product on modulation spaces and the discrete weighted Young's inequality}
is independent of all other theorems, here we bring it in advance to obtain Proposition 4.1.
\end{proof}
\begin{proposition}\label{Continuous weighted Young's inequality, near the origin}
Let $1\leq q, q_i \leq \infty, s_i\in \mathbb{R}$, for $i=1,2$.
Then
\begin{equation*}
L(q_1, s_1)\ast L(q_2, s_2) \subset L(q, s)
\end{equation*}
implies
\begin{equation*}
1+\frac{1}{q}\geq \frac{1}{q_1}+\frac{1}{q_2}.
\end{equation*}
\end{proposition}
\begin{proof}
For $0<a<1$, we define
\begin{equation*}
f(x)=g(x)=\chi_{B(0,a)}.
\end{equation*}
It is easy to verify that
\begin{equation*}
(f\ast g)(x)\gtrsim a^n
\end{equation*}
for $x\in B(0,\frac{a}{2})$.
Also, a direct calculation shows that
\begin{equation*}
\|f\|_{L(q_1,s_1)}\sim a^{{n}/{q_1}},~\|g\|_{L(q_2,s_2)}\sim a^{{n}/{q_2}}
\end{equation*}
and
\begin{equation*}
\|f\ast g\|_{L(q,s)}\gtrsim  a^{n+{n}/{q}}.
\end{equation*}
as $a\rightarrow 0$.
Hence, the inequality
\begin{equation*}
\|f \ast g\|_{L(q,s)}\lesssim \|f\|_{L(q_1,s_1)}\cdot \|g\|_{L(q_2,s_2)}
\end{equation*}
implies
\begin{equation*}
a^{n+{n}/{q}}\lesssim a^{{n}/{q_1}}\cdot a^{{n}/{q_2}}.
\end{equation*}
The conclusion of the proposition now trivially follows by letting $a \rightarrow 0$ in the above inequality.
\end{proof}
By Proposition \ref{Relationship between discrete weighted Young's inequality and continuous weighted Young's inequality},
Proposition \ref{Continuous weighted Young's inequality, near the origin}
and Theorem \ref{Sharpness of weighted Young's inequality, discrete form}, we now obtain the inclusion
\begin{equation*}
B \subset \bigcup_{i=1}^4B_i.
\end{equation*}
Next, we want to show the opposite inclusion
\begin{equation*}
\bigcup_{i=1}^{4}B_{i}\subset B\mathbf{.}\newline
\end{equation*}%
To this end, we need to show
\begin{equation*}
(\textbf{q},\textbf{s})\in \bigcup_{i=1}^4B_i\Longrightarrow L(q_1, s_1)\ast L(q_2, s_2) \subset L(q, s).
\end{equation*}
First, we show
\begin{equation}\label{for proof, 10}
(\textbf{q},\textbf{s})\in B_1\Longrightarrow L(q_1, s_1)\ast L(q_2, s_2) \subset L(q, s).
\end{equation}
Similar to the proof of Theorem \ref{Sharpness of weighted Young's inequality, discrete form}, we divide the proof into three cases.\\
Case 1: $s\geq 0, s_1\geq 0, s_2\geq 0;$ \\
Case 2: $s<0, s_1> 0, s_2>0;$\\
Case 3: $s<0, s_1\leq 0, s_2\geq 0$ or $s<0, s_1\geq 0, s_2\leq 0$.\\

Case 2 and Case 3 actually can be reduced to Case 1 with the following
arguments.

In Case 2, we choose
\begin{equation*}
\frac{1}{\bar{q}}=\max\left\{\frac{1}{q}+\frac{s}{n},~\frac{1}{q_1}+\frac{1}{q_2}-1,~0\right\}.
\end{equation*}
Then
\begin{equation*}
\frac{1}{q}\geq \frac{1}{\bar{q}},~
1+\frac{1}{\bar{q}}\geq \frac{1}{q_1}+\frac{1}{q_2}.
\end{equation*}
We can further choose $\bar{s}>0$ such that $1/q+s/n<1/\bar{q}+\bar{s}/n$, and then by Proposition \ref{Sharpness of embedding, continuous form}, we have
\begin{equation}\label{for proof, 14}
L(\bar{q},\bar{s})\subset L(q,s)
\end{equation}
and the new index $(\bar{\textbf{q}},\bar{\textbf{s}})=(\bar{q},q_1,q_2,\bar{s},s_1,s_2)$ belongs to Case 1.
So the conclusion in Case 2 follows from Case 1 and (\ref{for proof, 14}).

On the other hand, Case 3 can be reduced to Case 1 by a dual argument,
with a similar argument used in the proof of Theorem \ref{Sharpness of weighted Young's inequality, discrete form}.

Before we handle Case 1, we also need the following proposition which is
just a minor modification of Proposition \ref{For reduction, discrete form}.
So we omit its proof.
\begin{proposition}[For reduction, continuous form]\label{For reduction, continuous form}
Suppose $0< q, q_1, q_2\leq \infty$, $s>0$.
If
\begin{equation*}
\begin{cases}
s\leq s_1,~s\leq s_2,\\
\frac{1}{q}\leq\frac{1}{q_1},~\frac{1}{q}\leq\frac{1}{q_2},\\
1+\frac{1}{q}+\frac{s}{n}<\frac{1}{q_1}+\frac{s_1}{n}+\frac{1}{q_2}+\frac{s_2}{n},\\
1+\frac{1}{q}\geq \frac{1}{q_1}+\frac{1}{q_2},
\end{cases}
\end{equation*}
then we have
\begin{equation*}
L(q_1, s_1)\ast L(q_2, s_2) \subset L(q, s).
\end{equation*}
\end{proposition}
Now we proceed the proof in Case 1.

If $s>0$, we set
\begin{equation*}
\widetilde{q_i}=
\begin{cases}
q_i, &\text{if~} \frac{1}{q}\leq \frac{1}{q_i}
\\
q, &\text{if~} \frac{1}{q}> \frac{1}{q_i}
\end{cases}
\end{equation*}
for $i=1,2$. Then we choose $\ \widetilde{s_i}\ $ such that $\tilde{s_i}=s_i$ for $1/q\leqslant 1/q_i$,
\begin{equation*}
1/\tilde{q_i}+\tilde{s_i}/n< 1/q_i+s_i/n
\end{equation*}
when $1/q>1/q_i$,
and
\begin{equation*}
\begin{cases}
s\geq 0,~s\leq \widetilde{s_1},~s\leq \widetilde{s_2},\\
\frac{1}{q}\leq\frac{1}{\widetilde{q_1}},~\frac{1}{q}\leq\frac{1}{\widetilde{q_2}},\\
1+\frac{1}{q}+\frac{s}{n}<\frac{1}{\widetilde{q_1}}+\frac{\widetilde{s_1}}{n}+\frac{1}{\widetilde{q_2}}+\frac{\widetilde{s_2}}{n},\\
1+\frac{1}{q}\geq \frac{1}{\widetilde{q_1}}+\frac{1}{\widetilde{q_2}}.
\end{cases}
\end{equation*}
Using Proposition \ref{Sharpness of embedding, continuous form}, we obtain following embedding relations:
\begin{equation}\label{for proof, 15}
L(q_1,s_1)\subset L(\widetilde{q_1},\widetilde{s_1}),~L(q_2,s_2)\subset L(\widetilde{q_2},\widetilde{s_2}).
\end{equation}

The conclusion (\ref{for proof, 10}) follows from the embedding relations (\ref{for proof, 15}) and
Proposition \ref{For reduction, continuous form}.

If $s=0,~s_{1},s_{2}>0$, on can choose a small positive constant $\tilde{s}$,
such that the new index group $(q,q_1,q_2,\tilde{s},s_1,s_2)$ belongs to the previous case $s>0$.
Thus, we have $L(q_1, s_1)\ast L(q_2, s_2) \subset L(q,\tilde{s})\subset L(q, s)$.

If $s=s_1=0,~s_2>0$, we have
\begin{equation*}
\begin{cases}
s=s_1=0,~s_2>0,\\
1+\frac{1}{q}<\frac{1}{q_1}+\frac{1}{q_2}+\frac{s_2}{n},\\
\frac{1}{q}\leq \frac{1}{q_1},~\frac{1}{q}<\frac{1}{q_2}+\frac{s_2}{n},\\
1+\frac{1}{q}\geq\frac{1}{q_1}+\frac{1}{q_2}.
\end{cases}
\end{equation*}

Choose
\begin{equation*}
\frac{1}{r}=1+\frac{1}{q}-\frac{1}{q_1},
\end{equation*}
and observe that
\begin{equation*}
\frac{1}{q_2}\leq \frac{1}{r}\leq 1,~\frac{1}{r}<\frac{1}{q_2}+\frac{s_2}{n}.
\end{equation*}
Clearly, $r\in [1,\infty]$.
We then use Young's inequality and the embedding relation $L(q_2,s_2)\subset L^r$
to deduce that
\begin{equation*}
\begin{split}
\|f\ast g\|_{L^q}\lesssim &\|f\|_{L^{q_1}}\cdot\|g\|_{L^r}
\\
\lesssim &
\|f\|_{L^{q_1}}\cdot \|g\|_{L(q_2,s_2)}.
\end{split}
\end{equation*}
The case $s=s_2=0,~s_1>0$ can be handled similarly.
We have now completed the proof for (\ref{for proof, 10}).

The proof of
\begin{equation*}
(\textbf{q},\textbf{s})\in B_2\Longrightarrow L(q_1, s_1)\ast L(q_2, s_2) \subset L(q, s)
\end{equation*}
is a trivial one, we omit the details.

The proof for
\begin{equation*}
(\textbf{q},\textbf{s})\in B_3\Longrightarrow L(q_1, s_1)\ast L(q_2, s_2) \subset L(q, s)
\end{equation*}
can proceed following the same method used in the proof of $A\cap X_3=A_3$, we also omit the details.

Finally, we show
\begin{equation*}
(\textbf{q},\textbf{s})\in B_4\Longrightarrow L(q_1, s_1)\ast L(q_2, s_2) \subset L(q, s).
\end{equation*}
If $q=\infty$ or $q_1=1$ or $q_2=1$,
we can get the conclusion by the same method used in the proof of $A\cap X_4=A_4$ and Proposition \ref{Integral capability of weight, continuous form}. We only give the proof for $q=\infty$, since the other cases can be handled by a dual argument.
If $q=\infty$, we have $s>0, s<s_1, s<s_2$.
Take
\begin{equation*}
\frac{1}{t}=1-\frac{1}{q_1}-\frac{1}{q_2}=\frac{s_1+s_2-s}{n}.
\end{equation*}
Using Proposition \ref{Integral capability of weight, continuous form}, we deduce that
\begin{equation*}
  \begin{split}
    |\int_{\mathbb{R}^n}&f(x-y)g(y)dy\langle x\rangle^s|\\
    = &
    |\int_{\mathbb{R}^n}f(x-y)\langle x-y\rangle^{s_1}g(y)\langle y\rangle^{s_2}
    \frac{\langle x\rangle^s}{\langle x-y\rangle^{s_1}\langle y\rangle^{s_2}}dy|
    \\
    \lesssim &
    \|f\|_{l(q_1,s_1)}\|g\|_{L(q_2,s_2)}\left\|\frac{\langle x\rangle^s}{\langle x-y\rangle^{s_1}\langle y\rangle^{s_2}}\right\|_{L_y^t}
    \\
    \lesssim &
    \|f\|_{l(q_1,s_1)}\|g\|_{L(q_2,s_2)}.
  \end{split}
\end{equation*}
It implies that $\|f\ast g\|_{L(\infty,s)}\lesssim \|f\|_{l(q_1,s_1)}\|g\|_{L(q_2,s_2)}.$

Next, we consider the case for $q\neq \infty, q_1\neq 1,q_2\neq 1$.
By the symmetry of $\ s, s_1, s_2$, it suffices to handle the case $s_1, s_2\geq 0.$
In this case, we have
\begin{equation*}
|x|^{s_1}\leq \langle x\rangle^{s_1},~|x|^{s_2}\leq \langle x\rangle^{s_2},
\end{equation*}
and
\begin{equation*}
\|f\|_{\mathbb{L}(q_1,s_1)}\leq \|f\|_{L(q_1,s_1)},~\|g\|_{\mathbb{L}(q_2,s_2)}\leq \|g\|_{L(q_2,s_2)}.
\end{equation*}
On the other hand, we write
\begin{equation*}
\begin{split}
\|f\ast g\|_{L(q,s)}\leq
\bigg(&\int_{B(0,1)}|(f\ast g)(x)|^q\langle x \rangle^{sq} dx\bigg)^{{1}/{q}}
\\
&+\bigg(\int_{(B(0,1))^c}|(f\ast g)(x)|^q\langle x \rangle^{sq} dx\bigg)^{{1}/{q}}
\\
=:&
\widetilde{I}+\widetilde{I\!I}.
\end{split}
\end{equation*}
Using Theorem A, we have
\begin{equation*}
\begin{split}
\widetilde{I\!I}\sim \bigg(\int_{B^c(0,1)}|(f\ast g)(x)|^q|x|^{sq} dx\bigg)^{{1}/{q}}
\lesssim &\|f \ast g\|_{\mathbb{L}(q,s)}
\\
\lesssim &
\|f\|_{\mathbb{L}(q_1,s_1)}\cdot \|g\|_{\mathbb{L}(q_2,s_2)}
\\
\lesssim &
\|f\|_{L(q_1,s_1)}\cdot \|g\|_{L(q_2,s_2)}.
\end{split}
\end{equation*}
For $\widetilde{I}$, if
\begin{equation*}
1+\frac{1}{q}>\frac{1}{q_1}+\frac{1}{q_2},
\end{equation*}
one can choose
\begin{equation*}
0<\frac{1}{\widetilde{q}}<\frac{1}{q},
\end{equation*}
so that
\begin{equation*}
\begin{split}
\widetilde{I}
\lesssim
\bigg(\int_{B(0,1)}|(f\ast g)(x)|^{\widetilde{q}}\langle x \rangle^{\widetilde{sq}} dx\bigg)^{{1}/{\widetilde{q}}}
\leq
\|f\ast g\|_{L(\widetilde{q},s)},
\end{split}
\end{equation*}
and the new index
$(\widetilde{\textbf{q}},\widetilde{\textbf{s}})=(\widetilde{q},q_1,q_2,s,s_1,s_2)$ belongs to $B_1$.
By the fact that $(\textbf{q},\textbf{s})\in B_1\Longrightarrow L(q_1, s_1)\ast L(q_2, s_2) \subset L(q, s)$, we deduce the inequality
\begin{equation*}
\|f\ast g\|_{L(\widetilde{q},s)}
\lesssim
\|f\|_{L(q_1,s_1)}\cdot \|g\|_{L(q_2,s_2)}.
\end{equation*}
Combining $\widetilde{I}$ with $\widetilde{I\!I}$, we obtain
\begin{equation*}
\|f\ast g\|_{L(q,s)}
\lesssim
\widetilde{I}+\widetilde{I\!I}
\lesssim
\|f\|_{L(q_1,s_1)}\cdot \|g\|_{L(q_2,s_2)}.
\end{equation*}
If
\begin{equation*}
1+\frac{1}{q}=\frac{1}{q_1}+\frac{1}{q_2},
\end{equation*}
we have $s=s_1+s_2$, then $s=s_1=s_2=0$. It is a trivial case.

We complete the proof of Theorem \ref{Sharpness of weighted Young's inequality, continuous form}.

\section{Application on modulation spaces---Proof of Theorem
\ref{Sharpness of product on modulation spaces, bilinear case} and \ref{Sharp conditions for the multi-linear Fourier multipliers on modulation spaces}}

\subsection{The proof of Theorem \ref{Sharpness of product on modulation spaces, bilinear case}.}$~$\\
We first show the relationship between the product on modulation spaces and the discrete weighted Young's inequality.
We need to establish the following proposition.
\begin{proposition}[Relationship between the product on modulation spaces and the discrete weighted Young's inequality]\label{Relationship between the product on modulation spaces and the discrete weighted Young's inequality}
Let $1\leq p, q, p_i,q_i \leq \infty, s_i\in \mathbb{R}$, for $i=1,2$.
Then
\begin{equation*}
\|fg\|_{M_{p,q}^s}\lesssim \|f\|_{M_{p_1,q_1}^{s_1}}\|g\|_{M_{p_2,q_2}^{s_2}}
\end{equation*}
holds for $f,g \in \mathscr{S}(\mathbb{R}^{n})$ if and only if
\begin{equation*}
\frac{1}{p}\leq \frac{1}{p_1}+\frac{1}{p_2}
\end{equation*}
and
\begin{equation*}
l(q_1, s_1)\ast l(q_2, s_2) \subset l(q, s).
\end{equation*}
\end{proposition}
\begin{proof} We first show the necessity part.
Recalling that $\{\sigma_k\}_k$ is a smooth decomposition of $\mathbb{R}^n$ as defined in Section 2.
We can choose a function  $\hat{h}(\xi)$ in $C_c^{\infty}(\mathbb{R}^n)$ such that $\hat{h}\subset B(0,{1}/{8})$,
and that $\| h\|_{L^{p_1}}, \| h\|_{L^{p_2}}$ and $\| h\ast h\|_{L^{p}}$ are not equal to 0.
Define
\begin{equation*}
\hat{v}_{\lambda}=\hat{h}(\xi/\lambda),~0\leqslant \lambda\leqslant 1.
\end{equation*}
Observing that $\Box_0v_{\lambda}=v_{\lambda}$ and $\Box_kv_{\lambda}=v_{\lambda}$ for $k\neq 0$.
Then
\begin{equation*}
\|v_{\lambda}\|_{M_{p,q}^s}
=
\left(\sum_{k\in \mathbb{Z}^n}\|\Box_kv_{\lambda}\|^q_{L^p}\right)^{1/q}
=
\|\Box_0v_{\lambda}\|_{L^p}
=
\|v_{\lambda}\|_{L^p}.
\end{equation*}
Similarly, we obtain
\begin{equation*}
\|v_{\lambda}\cdot v_{\lambda}\|_{M_{p,q}^s}
=
\|v_{\lambda}\cdot v_{\lambda}\|_{L^p}.
\end{equation*}

Then we use the assumption
\begin{equation*}
\|v_{\lambda}\cdot v_{\lambda}\|_{M_{p,q}^s}\lesssim \|v_{\lambda}\|_{M_{p_1,q_1}^{s_1}}\|v_{\lambda}\|_{M_{p_2,q_2}^{s_2}}
\end{equation*}
to deduce
\begin{equation*}
  \begin{split}
    \lambda^{n(2-{1}/{p})}\|h^2\|_{L^p}\sim &\|v_{\lambda}\cdot v_{\lambda}\|_{L^p}
    \\
    \lesssim &
    \|v_{\lambda}\|_{L^{p_1}}\|v_{\lambda}\|_{L^{p_2}}
    \sim
    \lambda^{n(2-1/p_1-1/p_2)}.
  \end{split}
\end{equation*}
Letting $\lambda \rightarrow 0$, we obtain
\begin{equation*}
\frac{1}{p}\leq \frac{1}{p_1}+\frac{1}{p_2}.
\end{equation*}
Denote
\begin{equation*}
\hat{h_k}(\xi)=\hat{h}(\xi-k).
\end{equation*}
Let $a_k,b_k\in \mathbb{R}^{+}$ for $k\in \mathbb{Z}^n$.
By assuming that the following two series converge, we define two functions
\begin{equation*}
f(x)=\sum_{k\in \mathbb{Z}^n}a_kh(x)e^{2\pi i k\cdot x},\hspace{6mm}g(x)=\sum_{k\in \mathbb{Z}^n}b_kh(x)e^{2\pi i k\cdot x}.
\end{equation*}
We have
\begin{equation*}
(fg)(x)=\sum_{j,l\in \mathbb{Z}^n}a_jb_lh^2(x) e^{2\pi i (j+l)\cdot x}.
\end{equation*}
Observing that
\begin{eqnarray*}
\Vert \Box _{k}(fg)\Vert _{L^{p}} &=&\Vert
\sum_{j+l=k}a_{j}b_{l}h^{2}(x)e^{2\pi i(j+l)\cdot x}\Vert _{L^{p}(dx)}\simeq
\sum_{j+l=k}a_{j}b_{l}, \\
\Vert \Box _{k}(f)\Vert _{L^{p_1}} &\simeq &a_{k},\text{ \ }\Vert \Box
_{k}(g)\Vert _{L^{p_2}}\simeq b_{k}\text{ \ for all \ }k\in \mathbb{Z}^{n},
\end{eqnarray*}
by the definition of modulation space, we use the assumption
\begin{equation*}
\|fg\|_{M_{p,q}^s}\lesssim \|f\|_{M_{p_1,q_1}^{s_1}}\|g\|_{M_{p_2,q_2}^{s_2}}
\end{equation*}
to deduce
\begin{equation*}
\|\{a_k\}\ast \{b_k\}\|_{l(q,s)}\lesssim \|\{a_k\}\|_{l(q_1,s_1)}\|\{b_k\}\|_{l(q_2,s_2)}.
\end{equation*}

We turn to show the sufficiency of the proposition. Using the almost
orthogonality of the frequency projections $\sigma _{k}$, we have that for all  \ $k\in \mathbb{Z}^{n}$,
\begin{equation*}
\Box_k(fg)=\sum_{i,j\in \mathbb{Z}^n}\Box_k(\Box_if\cdot \Box_jg)
=\sum_{|l|\leq c(n)}\sum_{i+j=k+l}\Box_k(\Box_if\cdot \Box_jg),
\end{equation*}
where \ $c(n)$ \ is a constant depending only on \ $n.$
By the fact that $\Box_k$ is an $L^p$ multiplier, we use H\"{o}lder's inequality and Lemma \ref{embedding of Lp with Fourier compact support} to deduce that
\begin{equation}\label{for proof, 16}
\begin{split}
\|\Box_k(fg)\|_{L^p}\lesssim & \|\Box_k(fg)\|_{L^r}
\\
=&\|\sum_{|l|\leq c(n)}\sum_{i+j=k+l}\Box_k(\Box_if\cdot \Box_jg)\|_{L^r}
\\
\leq &
\sum_{|l|\leq c(n)}\sum_{i+j=k+l}\|\Box_if\|_{L^{p_1}}\|\Box_jg\|_{L^{p_2}},
\end{split}
\end{equation}
where $1/r=1/p_1+1/p_2$.
Then the discrete weighted Young's inequality implies that
\begin{equation}\label{for proof, 17}
\begin{split}
\|fg\|_{M_{p,q}^s}
= &
\big\|\{\|\Box_k(fg)\|_{L^p}\}\big\|_{l(q,s)}
\\
\lesssim &
\big\|\{\|\Box_if\|_{L^{p_1}}\}\ast \{\|\Box_jg\|_{L^{p_2}}\}\big\|_{l(q,s)}
\\
\lesssim &
\big\|\{\|\Box_kf\|_{L^{p_1}}\}\big\|_{l(q_1,s_1)}
\cdot
\big\|\{\|\Box_kg\|_{L^{p_2}}\}\big\|_{l(q_2,s_2)}
\\
=&
\|f\|_{M_{p_1,q_1}^{s_1}}\|g\|_{M_{p_2,q_2}^{s_2}}.
\end{split}
\end{equation}
Proposition \ref{Relationship between the product on modulation spaces and the discrete weighted Young's inequality} is proved.
\end{proof}
Now Theorem \ref{Sharpness of product on modulation spaces, bilinear case} can be verified immediately by the above
Proposition \ref{Relationship between the product on modulation spaces and the discrete weighted Young's inequality}
together with Theorem \ref{Sharpness of weighted Young's inequality, discrete form}.

\subsection{The proof of Theorem \ref{Sharp conditions for the multi-linear Fourier multipliers on modulation spaces}}$~$\\

To prove the necessity part, we take the special bilinear Fourier multiplier
$T(f,g)=fg$. Choose
\begin{equation*}
p=1,~p_1=2,~p_2=2.
\end{equation*}
By H\"{o}lder's inequality,
\begin{equation*}
\|fg\|_{L^p}\lesssim \|f\|_{L^{p_1}}\|g\|_{L^{p_2}}.
\end{equation*}
Hence, using the assumption of $(\mathbf{q},\mathbf{s})$ we have
\begin{equation*}
\|T(f,g)\|_{M_{p,q}^s}=\|fg\|_{M_{p,q}^s}\lesssim \|f\|_{M_{p_1,q_1}^{s_1}}\|g\|_{M_{p_2,q_2}^{s_2}}.
\end{equation*}
Then the conclusion follows by Theorem \ref{Sharpness of product on modulation spaces, bilinear case}.

For the sufficiency part, we notice that
\begin{equation*}
\Box_k\big(T(f,g)\big)=\sum_{i,j\in \mathbb{Z}^n}\Box_k\big(T(\Box_if, \Box_jg)\big)
=\sum_{|l|\leq c(n)}\sum_{i+j=k+l}\Box_k\big(T(\Box_if, \Box_jg)\big).
\end{equation*}
Recall the assumption that $T: L^{p_1}\times L^{p_2}\rightarrow L^p$, then
\begin{equation*}
\|T(\Box_if, \Box_jg)\|_{L^p}\lesssim \|\Box_if\|_{L^{p_1}}\|\Box_jg\|_{L^{p_2}}.
\end{equation*}
Using the fact that $\Box _{k}$ is an $L^{p}$ multiplier, we deduce that
\begin{equation*}
\begin{split}
\|\Box_k\big(T(f,g)\big)\|_{L^p}=&\|\sum_{|l|\leq c(n)}\sum_{i+j=k+l}\Box_k\big(T(\Box_if, \Box_jg)\big)\|_{L^p}
\\
\leq &
\sum_{|l|\leq c(n)}\sum_{i+j=k+l}\|\Box_if\|_{L^{p_1}}\|\Box_jg\|_{L^{p_2}}
\end{split}
\end{equation*}
as in the proof of (\ref{for proof, 16}).
The rest part of this proof is the same as the proof of (\ref{for proof, 17}).

\section{Fractional integral operators:
Proof of Theorem \ref{Sharpness of fractional integration operator, discrete form} and Theorem \ref{Sharpness of fractional integration operator, power weight}.}
\subsection{Proof of Theorem \ref{Sharpness of fractional integration operator, discrete form}}

Firstly, we introduce some notations for simplicity.
Denote
\begin{equation*}
C=\{(q, p, s, t) \in [1,\infty]^2\times \mathbb{R}^2:~\mathcal {I}_{\lambda}: l(p,t)\rightarrow l(q,s) \}.
\end{equation*}
Use $C_i$ to denote the set of all $(q, p, s, t) \in [1,\infty]^2\times \mathbb{R}^2$ satisfying condition $\mathcal {C}_i$
mentioned in Theorem \ref{Sharpness of fractional integration operator, discrete form} respectively, $i=1,3,4$.
Let $C_2=\emptyset$.
We use $Z$ to denote the set of all $(\mathbf{q},\mathbf{s})\in [1,\infty]^3\times \mathbb{R}^3$ satisfying
\begin{equation*}
\begin{cases}
s\leq t,\\
\frac{\lambda}{n}+\Big(\frac{1}{q}+\frac{s}{n}\Big)\vee 0\leq \Big(\frac{1}{p}+\frac{t}{n}\Big)\vee 0,\\
\frac{\lambda}{n}+\frac{1}{q}+\frac{s}{n}\leq 1,\\
(q', -s)=(1, \lambda-n)~\text{if}~ \frac{\lambda}{n}+\frac{1}{q}+\frac{s}{n}=1,\\
(p, t)=(1, \lambda-n)~\text{if}~ \frac{\lambda}{n}=\frac{1}{p}+\frac{t}{n}.
\end{cases}
\end{equation*}
We use $Z_1$ to denote the set of all $(\mathbf{q},\mathbf{s})\in Z$ satisfying
\begin{equation*}
\frac{\lambda}{n}+\Big(\frac{1}{q}+\frac{s}{n}\Big)\vee 0< \Big(\frac{1}{p}+\frac{t}{n}\Big)\vee 0.
\end{equation*}
Then $Z_1=C_1.$
We also use $Z_i ~(i=2, 3, 4)$ to denote the sets of all $(\mathbf{q},\mathbf{s})\in Z$ satisfying
\begin{equation*}
\frac{\lambda}{n}+\Big(\frac{1}{q}+\frac{s}{n}\Big)\vee 0=\Big(\frac{1}{p}+\frac{t}{n}\Big)\vee 0,
\end{equation*}
and
\begin{equation*}
\begin{cases}
\frac{1}{q}+\frac{s}{n}=0, ~\text{if}~i=2,\\
\frac{1}{q}+\frac{s}{n}<0, ~\text{if}~i=3,\\
\frac{1}{q}+\frac{s}{n}>0, ~\text{if}~i=4.
\end{cases}
\end{equation*}
To prove the theorem, we will use the same strategy as before. We first show
$C\subset Z$, then verify $C\cap Z_{j}=C_{j}$ for $j=1,2,3,4$. So, the
conclusion follows from the easy fact
\begin{equation*}
C=C\cap Z=C\cap \left( \bigcup_{i=1}^{4}Z_{i}\right)
=\bigcup_{i=1}^{4}(C\cap Z_{i})=\bigcup_{i=1}^{4}C_{i}.
\end{equation*}

To prove $C\subset Z$\textbf{, }for each positive integer $N,$ choosing
\begin{equation*}
f_N(k)=
\begin{cases}
1,&\text{for~}|k|\leq 2N
\\
0,&\text{otherwise},
\end{cases}
\end{equation*}
we then have
\begin{equation*}
\begin{split}
\sum_{j\in \mathbb{Z}^n, j\neq k}\frac{f_N(j)}{|k-j|^{n-\lambda}}
\geq &
\sum_{|j|\leq 2N, j\neq k}\frac{f_N(j)}{|k-j|^{n-\lambda}}
\\
\geq &
\sum_{|j|\leq 2N, j\neq k}\frac{1}{|k-j|^{n-\lambda}}
\\
\geq &
\sum_{|j|\leq N, j\neq 0}\frac{1}{|j|^{n-\lambda}}\sim N^{\lambda}
\end{split}
\end{equation*}
for $|k|\leq N$.
So we get
\begin{equation*}
\|\mathcal {I}_{\lambda}(f_N)\|_{l(q,s)}\gtrsim N^{\lambda}\bigg(\sum_{|k|\leq N}\langle k\rangle^{sq}\bigg)^{{1}/{q}}.
\end{equation*}
By the similar argument used in Subsection 3.2, we can use the inequalities
\begin{equation*}
N^{\lambda}\bigg(\sum_{|k|\leq N}\langle k\rangle^{sq}\bigg)^{{1}/{q}}
\lesssim \|f_N\|_{l(p,t)}
\lesssim \bigg(\sum_{|k|\leq 2N}\langle k\rangle^{tp}\bigg)^{{1}/{p}}
\end{equation*}
to deduce
\begin{equation*}
\frac{\lambda}{n}+\Big(\frac{1}{q}+\frac{s}{n}\Big)\vee 0\leq \Big(\frac{1}{p}+\frac{t}{n}\Big)\vee 0.
\end{equation*}
On the other hand,
for nonnegative $f$, we use
\begin{equation*}
(\mathcal {I}_{\lambda}f)(k)=\sum_{j\in \mathbb{Z}^n, j\neq k}\frac{f(j)}{|k-j|^{n-\lambda}}\gtrsim f(k-1),
\end{equation*}
to deduce
\begin{equation*}
\|f\|_{l(q,s)}\lesssim \|\mathcal {I}_{\lambda}(f)\|_{l(q,s)}\lesssim \|f\|_{l(p,t)},
\end{equation*}
which implies
\begin{equation*}
l(p, t)\subset l(q, s).
\end{equation*}
Finally, H\"{o}lder's inequality yields that
\begin{equation*}
\sum_{k\in \mathbb{Z}^n}\sum_{j\in \mathbb{Z}^n, j\neq k}\frac{a_j}{|k-j|^{n-\lambda}}c_k
\lesssim \|\{a_k\}\|_{l(p,t)}\|\{c_k\}\|_{l(q',-s)}.
\end{equation*}
Now we take
$a_0=1$ and $a_k=0~(k \neq 0)$
to obtain
\begin{equation*}
\sum_{k\in \mathbb{Z}^n, k\neq 0}\frac{c_k}{|k|^{n-\lambda}}
\lesssim
\|\{c_k\}\|_{l(q',-s)},
\end{equation*}
which clearly implies
\begin{equation*}
l(q',-s)\subset l(1, \lambda-n).
\end{equation*}

Then, we take
$c_0=1$ and $c_k=0~(k \neq 0)$
to obtain
\begin{equation*}
l(p,t)\subset l(1, \lambda-n).
\end{equation*}
The other conditions of $Z$ then can be verified by the above three embedding
relations and Proposition \ref{Sharpness of embedding, discrete form}.

Next, we turn to show $C\cap Z_{j}=C_{j}$ for $j=1,2,3,4$.

First, to prove $C\cap Z_{1}=C_{1}$, we only need to show
that
\begin{equation*}
(q,p,s,t)\in C_{1}\Longrightarrow \mathcal{I}_{\lambda }:l(p,t)\rightarrow
l(q,s).
\end{equation*}
We divide this part of proof into two cases.\\
\textbf{Case 1:} ${\lambda}/{n}+{1}/{q}+{s}/{n}<1$.
Using a similar argument as used in Subsection 3.3, we only need to consider the case $s\geq 0$.
Thus we have
\begin{equation*}
\begin{cases}
0\leq s\leq t, \\
1+\frac{1}{q}+\frac{s}{n}<\frac{1}{p}+\frac{t}{n}+\left(1-\frac{\lambda}{n}\right), \\
\frac{1}{q}+\frac{s}{n}<1-\frac{\lambda}{n},~\frac{1}{q}+\frac{s}{n}<\frac{1}{p}+\frac{t}{n},~1<\frac{1}{p}+\frac{t}{n}+\left(1-\frac{\lambda}{n}\right).
\end{cases}
\end{equation*}
We can choose $q_1\in [1,\infty]$ such that
\begin{equation*}
\frac{1}{q}+\frac{s}{n}<\frac{1}{q_1}+\frac{s}{n}<1-\frac{\lambda}{n},
\end{equation*}
and
\begin{equation*}
\begin{cases}
0\leq s\leq t \\
1+\frac{1}{q}+\frac{s}{n}<\frac{1}{p}+\frac{t}{n}+\frac{1}{q_1}+\frac{s}{n} \\
\frac{1}{q}+\frac{s}{n}<\frac{1}{q_1}+\frac{s}{n},~\frac{1}{q}+\frac{s}{n}<\frac{1}{p}+\frac{t}{n},~1<\frac{1}{p}+\frac{t}{n}+\frac{1}{q_1}+\frac{s}{n}.
\end{cases}
\end{equation*}
By the fact
\begin{equation*}
\{\langle k\rangle^{\lambda-n}\}_{k\in \mathbb{Z}^n}\in l(q_1,s),
\end{equation*}
we use Theorem \ref{Sharpness of weighted Young's inequality, discrete form} to deduce
\begin{equation*}
\begin{split}
\|\mathcal {I}_{\lambda}(f)\|_{l(q,s)}\lesssim &
\|\{\langle k\rangle^{\lambda-n}\}\ast f\|_{l(q,s)}
\\
\lesssim&
\|\{\langle k\rangle^{\lambda-n}\}\|_{l(q_1,s)}\|f\|_{l(p,t)}
\lesssim
\|f\|_{l(p,t)}.
\end{split}
\end{equation*}
\textbf{Case 2:} $\frac{\lambda}{n}+\frac{1}{q}+\frac{s}{n}=1$.
In this case, we have $q=\infty, s=n-\lambda$.
\begin{equation*}
\begin{split}
\|\mathcal {I}_{\lambda}(f)\|_{l(\infty,n-\lambda)}= &
\sup_{k\in \mathbb{Z}^n}\left(\sum_{j\in \mathbb{Z}^n, j\neq k}\frac{f(j)}{|k-j|^{n-\lambda}}\langle k\rangle^{n-\lambda} \right)
\\
\lesssim &
\sup_{k\in \mathbb{Z}^n}\left(\sum_{j\in \mathbb{Z}^n}f(j)\langle j\rangle^t\frac{\langle k\rangle^{n-\lambda}}{\langle k-j\rangle^{n-\lambda}\langle j\rangle^t} \right)
\\
\lesssim &
\sup_{k\in \mathbb{Z}^n}\left\|\frac{\langle k\rangle^{n-\lambda}}{\langle k-j\rangle^{n-\lambda}\langle j\rangle^t}\right\|_{l_j^{p'}}\|f\|_{l(p,t)}.
\end{split}
\end{equation*}
Observing that
\begin{equation*}
0<n-\lambda=s\leq t
\end{equation*}
and
\begin{equation*}
1=\frac{\lambda}{n}+\frac{1}{q}+\frac{s}{n}<\frac{1}{p}+\frac{t}{n}\Rightarrow p'>\frac{n}{t},
\end{equation*}
we use Proposition \ref{Integral capability of weight, discrete form} to deduce
\begin{equation*}
\sup_{k\in \mathbb{Z}^n}\left\|\frac{\langle k\rangle^{n-\lambda}}{\langle k-j\rangle^{n-\lambda}\langle j\rangle^t}\right\|_{l_j^{p'}}\lesssim 1.
\end{equation*}
The desired conclusion $C\cap Z_{1}=C_{1}$ then follows.

Next, we claim $C\cap Z_{2}=\emptyset $.
In fact, we have a special embedding relationship in this case.
Firstly, we have ${1}/{p}+{t}/{n}={\lambda}/{n}$, then $p=1, t=\lambda-n$.
For positive integer $N,$ define
\begin{equation*}
a_{k,N}=
\begin{cases}
1,&\text{for~}|k|\leq 2N,\\
0,&\text{otherwise}.
\end{cases}
\end{equation*}
It is easy to check that
\begin{equation*}
\|\{a_{k,N}\}\|_{l(p,t)}\sim N^{{n}/{p}+t}=N^{\lambda}
\end{equation*}
and
\begin{equation*}
\begin{split}
\sum_{k\in \mathbb{Z}^n}\sum_{j\in \mathbb{Z}^n, j\neq k}\frac{a_{j,N}}{|k-j|^{n-\lambda}}c_k
\gtrsim &
\sum_{|k|\leq N}c_k\sum_{|j|\leq 2N, j\neq k}\frac{a_{j,N}}{|k-j|^{n-\lambda}}
\\
\gtrsim &
\sum_{|k|\leq N}c_k\sum_{|j|\leq N, j\neq 0}\frac{1}{|j|^{n-\lambda}}
\\
\gtrsim &
N^{\lambda}\sum_{|k|\leq N}c_k.
\end{split}
\end{equation*}
Then we use the assumption
\begin{equation*}
\sum_{k\in \mathbb{Z}^n}\sum_{j\in \mathbb{Z}^n, j\neq k}\frac{a_{j,N}}{|k-j|^{n-\lambda}}c_k
\lesssim
\|\{a_{k,N}\}\|_{l(p,t)}\|\{c_k\}\|_{l(q',-s)}
\end{equation*}
to deduce that
\begin{equation*}
N^{\lambda}\sum_{|k|\leq N}c_k
\lesssim
N^{\lambda}\|\{c_k\}\|_{l(q',-s)}
\end{equation*}
as $N\rightarrow \infty$. It gives the embedding
\begin{equation*}
l(q',-s)\subset l(1,0).
\end{equation*}
Using the fact ${1}/{q'}+{(-s)}/{n}=1$ and Proposition \ref{Sharpness of embedding, discrete form}, we obtain
\begin{equation*}
q=\infty, s=0,
\end{equation*}
which contradicts the fact
\begin{equation*}
t=\lambda-n, s\leq t.
\end{equation*}

To prove $C\cap Z_{3}=C_{3}$,\textbf{\ }we first easily see that
\begin{equation*}
C\cap Z_{3}\subset C_{3}
\end{equation*}
by the fact that $\frac{\lambda }{n}=\frac{1}{p}+\frac{t}{n}$, then $
p=1,t=\lambda -n.$

To verify $C_3\subset C\cap Z_3$, it is sufficient to show
\begin{equation*}
(q, p, s, t)\in C_3\Longrightarrow \mathcal {I}_{\lambda}: l(p,t) \rightarrow l(q,s).
\end{equation*}
However, this conclusion can be reduced to the $\mathcal{C}_{1}$ condition
case by a dual argument. In fact, $(p',q',-t,-s)$ belongs to $\mathcal {C}_1$.
It implies that
\begin{equation*}
  \mathcal {I}_{\lambda}: l(q',-s) \rightarrow l(p',-t).
\end{equation*}
The desired conclusion $\mathcal {I}_{\lambda}: l(p,t) \rightarrow l(q,s)$ follows by a dual argument.

Our last step is to show $C\cap Z_{4}=C_{4}$.

To prove
\begin{equation*}
C\cap Z_4\subset C_4,
\end{equation*}
we first verify
\begin{equation*}
(q, p, s, t)\subset C\cap Z_4
\Longrightarrow
\frac{\lambda}{n}+\frac{1}{q}+\frac{s}{n}\neq 1.
\end{equation*}
In fact, if ${\lambda}/{n}+{1}/{q}+{s}/{n}=1$, then we have
\begin{equation*}
\frac{1}{p'}+\frac{-t}{n}=0,~\frac{1}{q'}+\frac{-s}{n}=\frac{\lambda}{n}
\end{equation*}
which can be reduced to the proof of $C\cap Z_2= \emptyset$ by a dual argument.

Second, we need to show
\begin{equation*}
(q, p, s, t)\subset C\cap Z_4
\Longrightarrow
\frac{1}{q}\leq \frac{1}{p}.
\end{equation*}
We define
\begin{equation*}
f_N(k)=
\begin{cases}
\langle k \rangle^{-t-{n}/{p}},&\text{for~}|k|\leq N,
\\
0,&\text{otherwise}.
\end{cases}
\end{equation*}
It is easy to check that
\begin{equation*}
\|\{f_N(k)\}\|_{l(p,t)}=(\ln N)^{{1}/{p}}.
\end{equation*}
On the other hand
\begin{equation*}
\begin{split}
\|\mathcal {I}_{\lambda}(f_N)\|_{l(q,s)}= &
\left(\sum_{k\in \mathbb{Z}^n}\bigg(\sum_{j\in \mathbb{Z}^n, j\neq k}\frac{f_N(j)}{|k-j|^{n-\lambda}} \bigg)^q\langle k\rangle^{sq}\right)^{{1}/{q}}
\\
\gtrsim  &
\left(\sum_{|k|\leq |N|}\bigg(\sum_{|j-k|\leq \frac{|k|}{2}, j\neq k}\frac{f_N(j)}{|k-j|^{n-\lambda}} \bigg)^q\langle k\rangle^{sq}\right)^{{1}/{q}}
\\
\gtrsim  &
\left(\sum_{k\leq |N|}\bigg(\sum_{|j-k|\leq \frac{|k|}{2}, j\neq k}\frac{\langle k \rangle^{-t-\frac{n}{p}}}{|k-j|^{n-\lambda}} \bigg)^q\langle k\rangle^{sq}\right)^{{1}/{q}}
\\
\gtrsim  &
\left(\sum_{k\leq |N|}\bigg(\langle k \rangle^{\lambda}\langle k \rangle^{-t-\frac{n}{p}} \bigg)^q\langle k\rangle^{sq}\right)^{{1}/{q}}
\sim
(\ln N)^{{1}/{q}}.
\end{split}
\end{equation*}
Then, we deduce the inequality
\begin{equation*}
(\ln N)^{{1}/{q}}\lesssim (\ln N)^{{1}/{p}},
\end{equation*}
which implies the desired conclusion $1/q\leqslant 1/p$.\\
Finally, we want show that
\begin{equation*}
(q, p, s, t)\in C\cap Z_4\cap \left\{(q, p, s, t)\in Z_4: ~s=t \right\}
\end{equation*}
implies $q\neq \infty, p\neq 1.$

If $s=t, q=\infty$, we have
\begin{equation*}
\frac{\lambda}{n}=\frac{1}{p}.
\end{equation*}
For any $N\in \mathbb{Z}^n$, we define
\begin{equation*}
g_N(k)=
\frac{\langle k\rangle^{-t}}{\langle N-k\rangle^{\lambda}}\left(1+\ln\langle N-k\rangle\right)^{\epsilon}
\end{equation*}
for all $k\in \mathbb{Z}^n$, where
$\ \epsilon\ $
is a real number that satisfies $p\epsilon<-1$ and $\epsilon\geq -1$.
Now
\begin{equation*}
\begin{split}
\|\mathcal {I}_{\lambda}(g_N)\|_{l(\infty,s)}= &
\sup_{k\in \mathbb{Z}^n}\bigg(\sum_{j\in \mathbb{Z}^n, j\neq k}\frac{g_N(j)}{|k-j|^{n-\lambda}}\langle k\rangle^{s} \bigg)
\\
\gtrsim  &
\sum_{|j-N|\leq \frac{|N|}{2}, j\neq k}\frac{g_N(j)}{|N-j|^{n-\lambda}}\langle N\rangle^{s}
\\
\sim  &
\sum_{|j-N|\leq \frac{|N|}{2}, j\neq k}\frac{\langle j\rangle^{-t}}{\langle N-j\rangle^{\lambda}}\left(1+\ln\langle N-j\rangle\right)^{\epsilon}
\frac{1}{|N-j|^{n-\lambda}}\langle N\rangle^{s}
\\
\sim  &
\sum_{|j-N|\leq \frac{|N|}{2}, j\neq k}\frac{1}{\langle N-j\rangle^{n}}\left(1+\ln\langle N-j\rangle\right)^{\epsilon}
\rightarrow \infty
\end{split}
\end{equation*}
as $N\rightarrow \infty.$ This contradicts the fact that
\begin{equation*}
\|g_N\|_{l(p,t)}\lesssim 1
\end{equation*}
uniformly as $N\rightarrow \infty.$

The case $s=t, p=1$ can be handled by a dual argument.

For the proof of
\begin{equation*}
C_4\subset C\cap Z_4,
\end{equation*}
we only need to show
\begin{equation*}
(q, p, s, t)\in C_4\Longrightarrow \mathcal {I}_{\lambda}: l(p,t) \rightarrow l(q,s).
\end{equation*}
Based on the complete result (see Theorem 5 in \cite{Strichartz}),
using an adaptation of the implication method to the setting of boundedness of fractional integrals,
we get this conclusion. In the endpoint case $p=1$ or $q=\infty$, we can also use Proposition \ref{Integral capability of weight, discrete form}
to verify this conclusion, just like we handled the same case in the proof of Theorem \ref{Sharpness of weighted Young's inequality, discrete form}.

\subsection{Proof of Theorem \ref{Sharpness of fractional integration operator, power weight}}
Since the sufficiency has been obtained by Strichartz \cite{Strichartz},
we only show that
\begin{equation*}
I_{\lambda}: \mathbb{L}(p,t) \rightarrow \mathbb{L}(q,s) \Longrightarrow (q, p, s, t)\in C_4 .
\end{equation*}
By a dilation argument, we obtain
\begin{equation}\label{for proof, 6}
\frac{\lambda}{n}+\frac{1}{q}+\frac{s}{n}=\frac{1}{p}+\frac{t}{n}.
\end{equation}
Hence, the only thing that we need to clarify is:
\begin{equation*}
\frac{1}{q}+\frac{s}{n}\geq 0.
\end{equation*}
Set
\begin{equation*}
f(x)=\chi_{B(2,1)}(x).
\end{equation*}
Then we have
\begin{equation*}
\|f\|_{\mathbb{L}(p,t)}\lesssim 1.
\end{equation*}
On the other hand,
\begin{equation*}
\begin{split}
\|I_{\lambda}(f)\|_{\mathbb{L}(q,s)}
\gtrsim &
\|I_{\lambda}(f)\chi_{B(0,1)}\|_{\mathbb{L}(q,s)}
\\
\gtrsim &
\|\chi_{B(0,1)}\|_{\mathbb{L}(q,s)}.
\end{split}
\end{equation*}
So we should have
\begin{equation*}
\|\chi_{B(0,1)}(x)\|_{\mathbb{L}(q,s)}\lesssim 1,
\end{equation*}
which implies
\begin{equation}\label{for proof, 7}
\frac{1}{q}+\frac{s}{n}\geq 0.
\end{equation}
According to Theorem \ref{Sharpness of fractional integration operator, discrete form}, we can use an implication argument (as Proposition
\ref{Implication method}) to deduce
\begin{equation*}
(q, p, s, t)\in \bigcup_{i=1,2,4}C_i .
\end{equation*}
Combining (\ref{for proof, 6}) with (\ref{for proof, 7}), we obtain
\begin{equation*}
(q, p, s, t)\in C_4 .
\end{equation*}

Now, we complete the proof of Theorem \ref{Sharpness of fractional integration operator, power weight}.

\subsection*{Acknowledgements}
The authors sincerely appreciate the anonymous referee for checking this paper very carefully
and giving very detailed comments, which greatly improved this article.
This work was supported by the National Natural Sciences Foundation of China (Nos. 11371295, 11471041, 11471288, 11671414 and 11601456).

\end{document}